\documentclass[a4paper,11pt]{article}

\usepackage{hyperref} 
\usepackage{url}

\usepackage[english]{babel}

\usepackage{amsmath,amssymb} 
\usepackage{xcolor}
\usepackage{array}
\usepackage{rotating}
\usepackage{graphicx}

\usepackage{mystyle}
\usepackage{notations}

\DeclareGraphicsExtensions{.png,.pdf,.eps,.jpg} 
\graphicspath{{./images/}}

\hypersetup{  
   bookmarks=true,
   backref=true,
   pagebackref=false,
   colorlinks=true,
   linkcolor=blue,
   citecolor=red,
   urlcolor=blue,
   pdftitle={Sparse Spikes Superresolution with concentrated or arbitrary Sampling},
   pdfauthor={V. Duval},
   pdfsubject={Positive sparse Spikes Super-resolution}
}

\title{A characterization of the Non-Degenerate Source Condition in Super-Resolution}

\author{Vincent Duval\footnotemark[1]\ \footnotemark[2]}

\begin{document}

\maketitle

\footnotetext[2]{CEREMADE, Universit\'e Paris-Dauphine, Place du Marechal De Lattre De Tassigny, 75775 PARIS CEDEX 16, FRANCE.}
\footnotetext[1]{INRIA, Mokaplan}


\begin{abstract}
  In a recent article, Schiebinger \textit{et al.} provided sufficient conditions for the noiseless recovery of a signal made of $\M$ Dirac masses given $2\M+1$ observations of, \eg, its convolution with a Gaussian filter, using the Basis Pursuit for measures.
In the present work, we show that a variant of their criterion provides a necessary and sufficient condition for the Non-Degenerate Source Condition (NDSC) which was introduced by Duval and Peyr\'e to ensure support stability in super-resolution. We provide sufficient conditions which, for instance, hold unconditionally for the Laplace kernel provided one has at least $2\M$ measurements. For the Gaussian filter, we show that those conditions are fulfilled in two very different configurations: samples which approximate the uniform Lebesgue measure or, more surprisingly, samples which are all confined in a sufficiently small interval.
\end{abstract}

\begin{keywords}Super-resolution; $\ell^1$ norm; Total Variation minimization; LASSO for measures; T-Systems\end{keywords}

\begin{AMS}90C25, 68U10\end{AMS}


\section{Introduction}

Super-resolution imaging, or observing small structures below the diffraction limit, is a fundamental problem in optical science, concerning applications such as microscopy~\cite{mccutchen_superresolution_1967}, astronomy~\cite{puschmann2005super}, medical imaging~\cite{greenspan_super-resolution_2009}. It is also of great importance in radar sensing~\cite{odendaal_two-dimensional_1994} and geophysics~\cite{khaidukov_diffraction_2004}. While practical methods have made tremendous progress, the mathematical understanding of that problem is still limited. With the seminal works on sparse super-resolution~\cite{deCastro-exact2012,bredies-inverse2013,candes-towards2013} and the simultaneous emergence of ``off-the-grid'' methods for line spectral estimation~\cite{bhaskar-atomic2011,tang2013compressed}, several theoretical advances have been made in the recovery of sparse signals in a continuous domain. 

Typically, given a domain $\dompos\subseteq \RR$, one wants to recover the signal 
\begin{align}
  \mo \eqdef \sum_{i=1}^\M \amp_i\delta_{\x_i} 
\end{align}
where $\{\x_i\}_{i=1}^\M\subset\dompos$  are spikes locations and $\{\amp_i\}_{i=1}^\M\subset \CC$ are their amplitudes, but one only has access to the observation 
\begin{align}
  y=\sum_{i=1}^\M\amp_i \varphi(\x_i) + w,
\end{align}
where $\varphi(x)\in \Hh$ is the impulse response, or \textit{point spread function}  (in a generalized sense), $\Hh$ is a Hilbert space, and $w$ is some additive noise. For instance, one may consider $\dompos= [0,1]$, $\Hh=\CC^{K}$, and let $\varphi(x)=(c_{j_k}(m))_{1\leq k\leq K}$ be a vector containing the Fourier coefficients of $\delta_x$ for $K$ prescribed frequencies (see~\cite{candes-towards2013}). Alternatively, one may choose $\Hh=\Ldeux(\RR)$ and $\varphi(x)=\tilde{\psi}(\cdot-x)$ is the impulse response of a linear translation invariant filter~\cite{tang2013atomic}. A more realistic setting is to consider that the convolution has been sampled, which can be handled using the more general choice 
\begin{align}
  \label{eq:Hsample}
  \varphi(x): \s\mapsto \psi(\x,\s) \qandq \Hh=\Ldeux(\domobs,\Ps),
\end{align}
where $\psi$ encodes the impulse response (\eg $\psi(\x,\s)=\tilde{\psi}(\s-\x)$ for a convolution) and $\Ps$ is some (positive) measure (typically, but not necessarily, with finite support) on $\domobs\subseteq \RR$, which we call sampling measure.

A ground-breaking result of~\cite{candes-towards2013} is that a measure $\mo$ on $\dompos=[0,1]$ (with periodic boundary condition) can be recovered from its first $2f_c+1$ Fourier coefficients, \ie for 
  \begin{align}
    w=0,\quad \varphi(x)\eqdef \left(e^{-2\imath\pi k x}\right)_{-f_c\leq k\leq f_c} \qandq    y^{(0)}=\int_\dompos\varphi(x)\d\mo(x),
  \end{align}
  by solving the convex minimization problem
  \begin{align}
    \min_{m\in \Mm(\dompos)} \abs{\m}(\dompos) \quad \mbox{s.t.} \int_\dompos\varphi(x)\d\m(x)=y^{(0)},\label{eq:introtvmin}
  \end{align}
  provided $m_0$ is made of spikes which are at least separated by a distance of $C/f_c$ where $C\geq 1.26$ (see~\cite{fernandez2016super}). The minimization above is performed in $\Mm(\dompos)$, the space of all Radon measures on $\dompos$, and  $\abs{\m}(\dompos)$ denotes the dual norm of $\normi{\cdot}$ (also known as \textit{total variation}).
  
  That separation condition is in fact a fundamental limit of total variation recovery when looking for spikes with arbitrary sign (or phase), hence the term ``super-resolution'' in this context is arguable. Indeed, if one observes $K$ measurements, \ie $\varphi(x)=(\varphi_k(x))_{1\leq k\leq K}\in \CC^K$ with $\varphi_k$ smooth, the optimality of $m_0$ is equivalent to being able to interpolate its sign/phase  with a function of the form
  \begin{align}\label{eq:etaintro}
    \eta(x)=\sum_{k=1}^K \alpha_k \varphi_k(x), \quad \normi{\eta}\leq 1.
  \end{align}
  The span of $\{\varphi_k\}_{1\leq k\leq K}$ being finite-dimensional, there exists a constant $C>0$ such that the Bernstein inequality $\normi{\eta'}\leq C\normi{\eta}$ holds, which prevents $\eta$ to interpolate the values $-1$ and $+1$ at arbitrarily close locations. More quantitative bounds are provided in~\cite{tang2015resolution} for several families of functions, and in~\cite{2015-duval-focm} for Fourier measurements, explaining that~\eqref{eq:introtvmin} is unable to resolve opposite spikes that are too close to each other.

  However, the situation \textit{radically changes if all the amplitudes of $\mo$ are positive}. In that case, the authors of~\cite{deCastro-exact2012} notice that for impulse responses $\varphi=(\varphi_k)_{1\leq k\leq K}$ which satisfy some 
  $T$-system property~\cite{karlin1966tchebycheff}, $\mo$ is the unique solution to~\eqref{eq:introtvmin} regardless of the separation of its spikes. That phenomenon, which is investigated further in~\cite{schiebinger2015superresolution}, paves the way for \textit{effective superresolution using sparse convex methods}, considering variants which impose nonnegativity such as  
  \begin{align}
    \min_{m\in \Mm^+(\dompos)} \abs{\m}(\dompos) \quad \mbox{s.t.} \int_\dompos\varphi(x)\d\m(x)=y^{(0)}.\label{eq:intropostvmin}
  \end{align}
  Nevertheless, a critical issue in super-resolution problems is the stability to noise. The present paper discusses the stability to noise, in the sense of support stability,  of such reconstructions for arbitrarily close spikes when considering the regularized inverse problem
  \begin{align}\label{eq:lassoposintro}
  \min_{m\in \Mm^+(\dompos)} \la\abs{\m}(\dompos) +\frac{1}{2}\left\| \int_\dompos\varphi(x)\d\m(x) -y\right\|_\Hh^2,
\end{align}
given some noisy version $y=(y_k)_{1\leq k\leq K}$ of $y^{(0)}=\int_\dompos \varphi\d m_0$. As explained in~\cite{2015-duval-focm}, the crux of the matter is to  understand whether the \textit{vanishing derivatives precertificate}, a generalization of the Fuchs precertificate~\cite{fuchs2004on-sp}, is able to take the value $1$ on $\{\pos_i\}_{i=1}^{\M}$ while being less than $1$ in $\dompos\setminus\{\pos_i\}_{i=1}^{\M}$. We call this property the \textit{Non-Degenerate Source Condition}.
Building on the work~\cite{schiebinger2015superresolution}, we propose in Section~\ref{sec:tsystem} a necessary and sufficient condition for the Non-Degenerate Source Condition. We show that for some filters, that condition holds for any choice of $\domobs$ and $\Ps$  (that is, the repartition of the measurements and their respective weights in the $\Hh$-norm) provided there are at least $2\M$ samples, regardless of the separation of the points $\pos_1,\ldots,\pos_\M$. 

Such is the case of the Laplace filter (see Figure~\ref{fig:laplace}), as shown in Section~\ref{sec:laplacegauss}, which provides theoretical foundations for the techniques used in~\cite{Laplace} for microscopy imaging. Admittedly, the constructed dual certificates decay slowly when going away from the $\pos_i$'s, indicating that the stability constants are not so good, and the problem is numerically unstable at large scales. But reconstructing a measure from a few Laplace measurements is a fundamentally difficult problem, and any reconstruction method is bound to face similar limitations. Our results indicate that~\eqref{eq:lassoposintro} is able to provide a sound reconstruction method which is robust to small perturbations, at least for a small number of spikes. This property is confirmed by the numerical experiments of~\cite{Laplace}.

The case of the convolution with a Gaussian kernel is more subtle. We show in Section~\ref{sec:gauss} that~\eqref{eq:intropostvmin} and \eqref{eq:lassoposintro} yield a support-stable reconstructions provided that the convolution is fully observed (and $\Ps$ is the Lebesgue measure). In general, this property does not hold for any sampling measure $\Ps$, but it holds for discrete measures $\Ps$ which approximate the Lebesgue measure sufficiently well, see Figure~\ref{fig:gauss}.
Surprisingly, this support recovery property also holds in the case where the convolution is sampled in a very small interval. We find this ``confined sampling'' setting quite counter-intuitive, as the samples can be placed anywhere (provided they are in some small interval) with respect to the spikes to recover, contrary \eg to~\cite{bernstein2017deconvolution} which requires at least two samples in the neighborhood each spike to recover.

\begin{figure}
	\centering  
  {\includegraphics[width=0.32\linewidth,clip,trim=1cm 0.5cm 1cm 0.9cm]{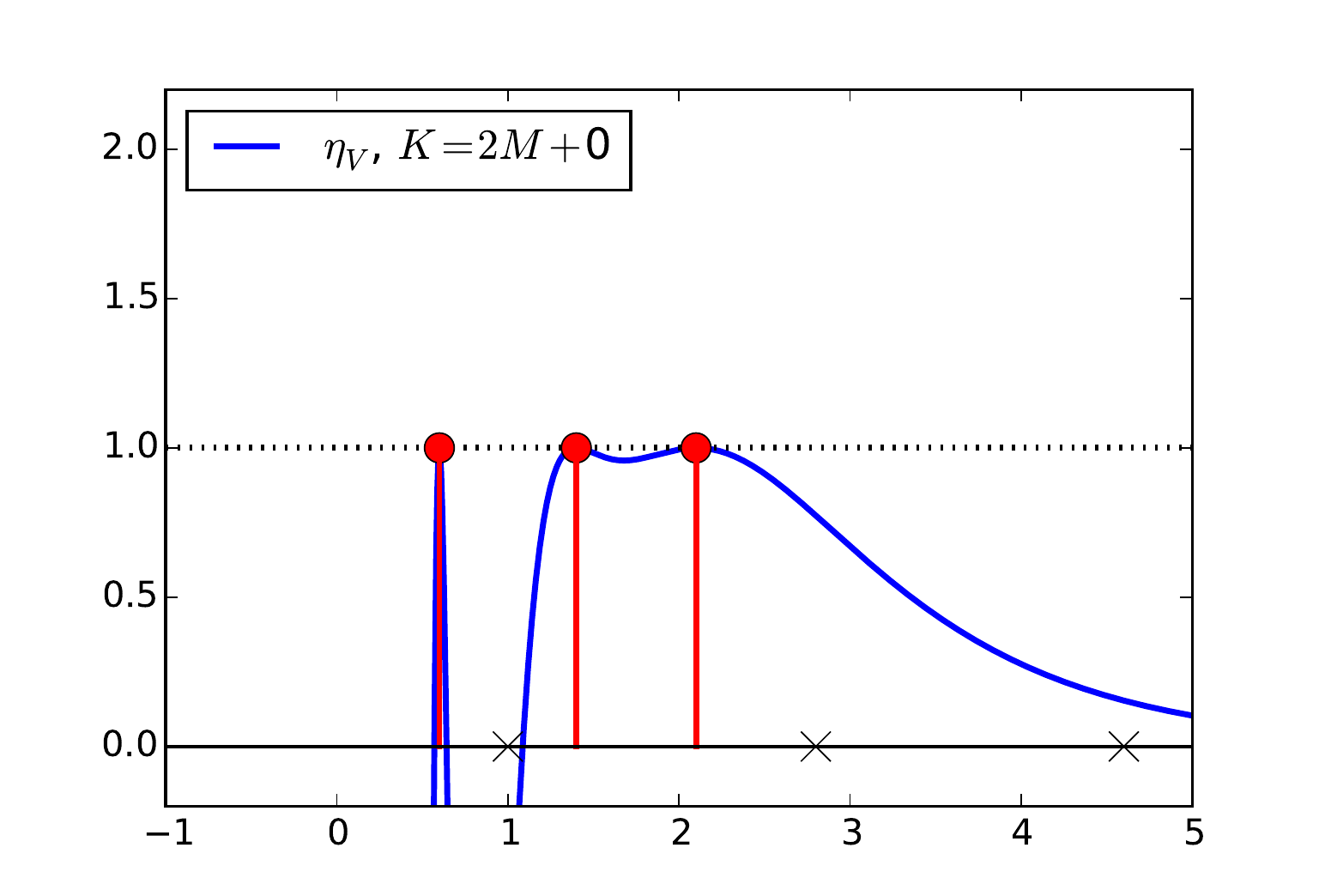}} 
  {\includegraphics[width=0.32\linewidth,clip,trim=1cm 0.5cm 1cm 0.9cm]{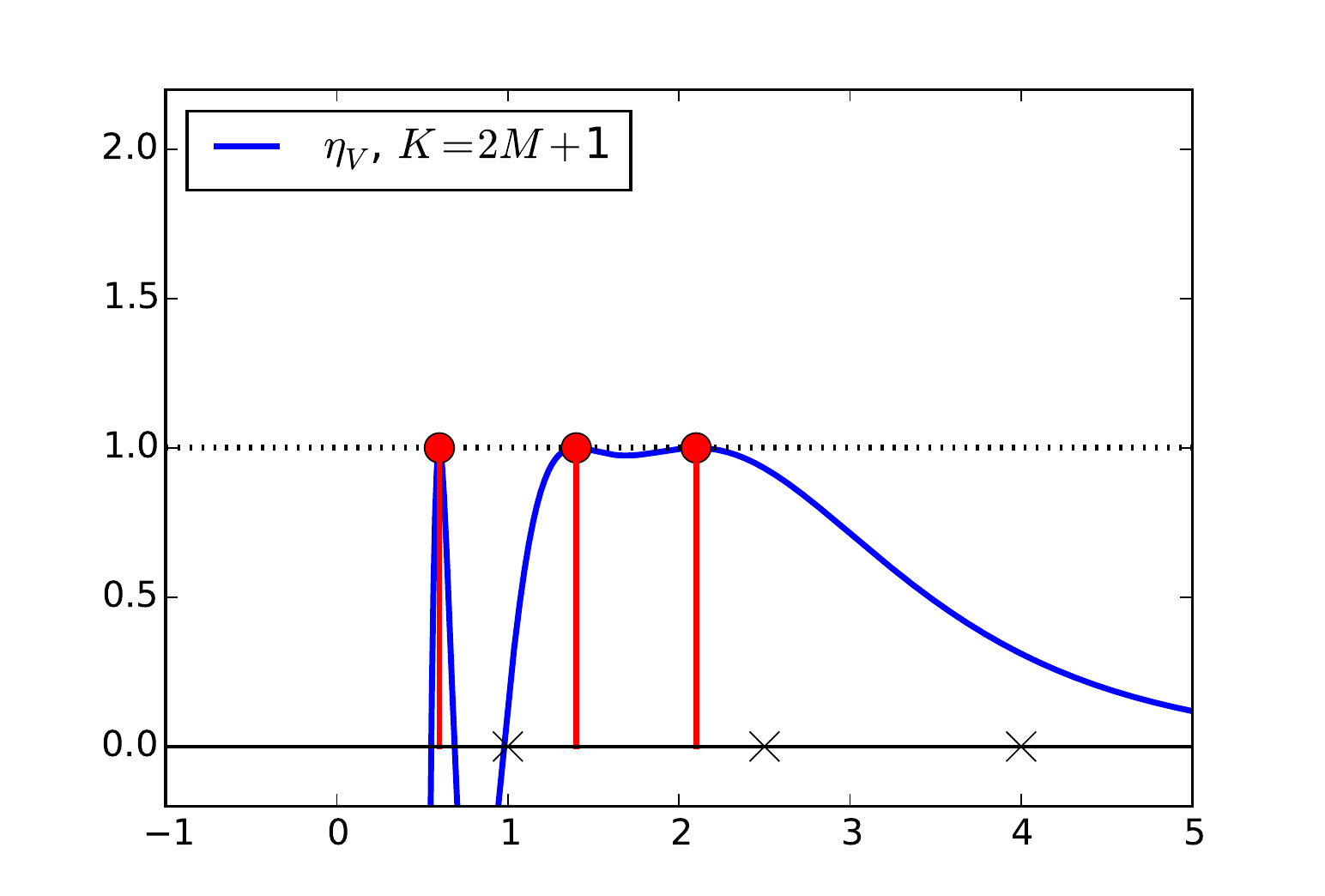}} 
  {\includegraphics[width=0.32\linewidth,clip,trim=1cm 0.5cm 1cm 0.9cm]{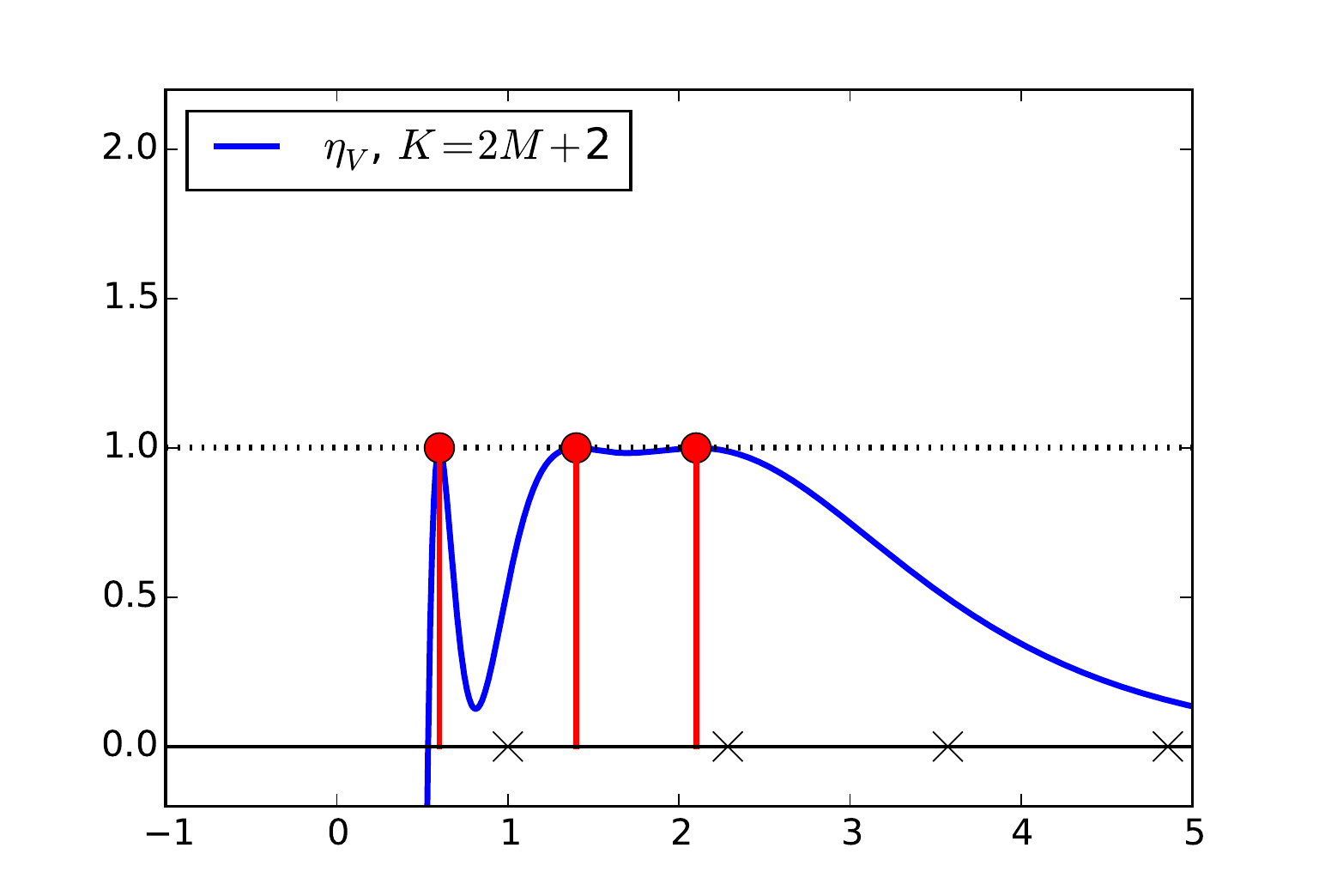}} 
  \caption{\label{fig:laplace}Vanishing derivatives precertificate for Laplace observations ($K=2M$, $2M+1$, and $2M+2$ measurements). It is sufficient to observe $K=2\M$ samples.}
\end{figure}

\begin{figure}
	\centering  
  {\includegraphics[width=0.32\linewidth,clip,trim=1cm 0.5cm 1cm 0.9cm]{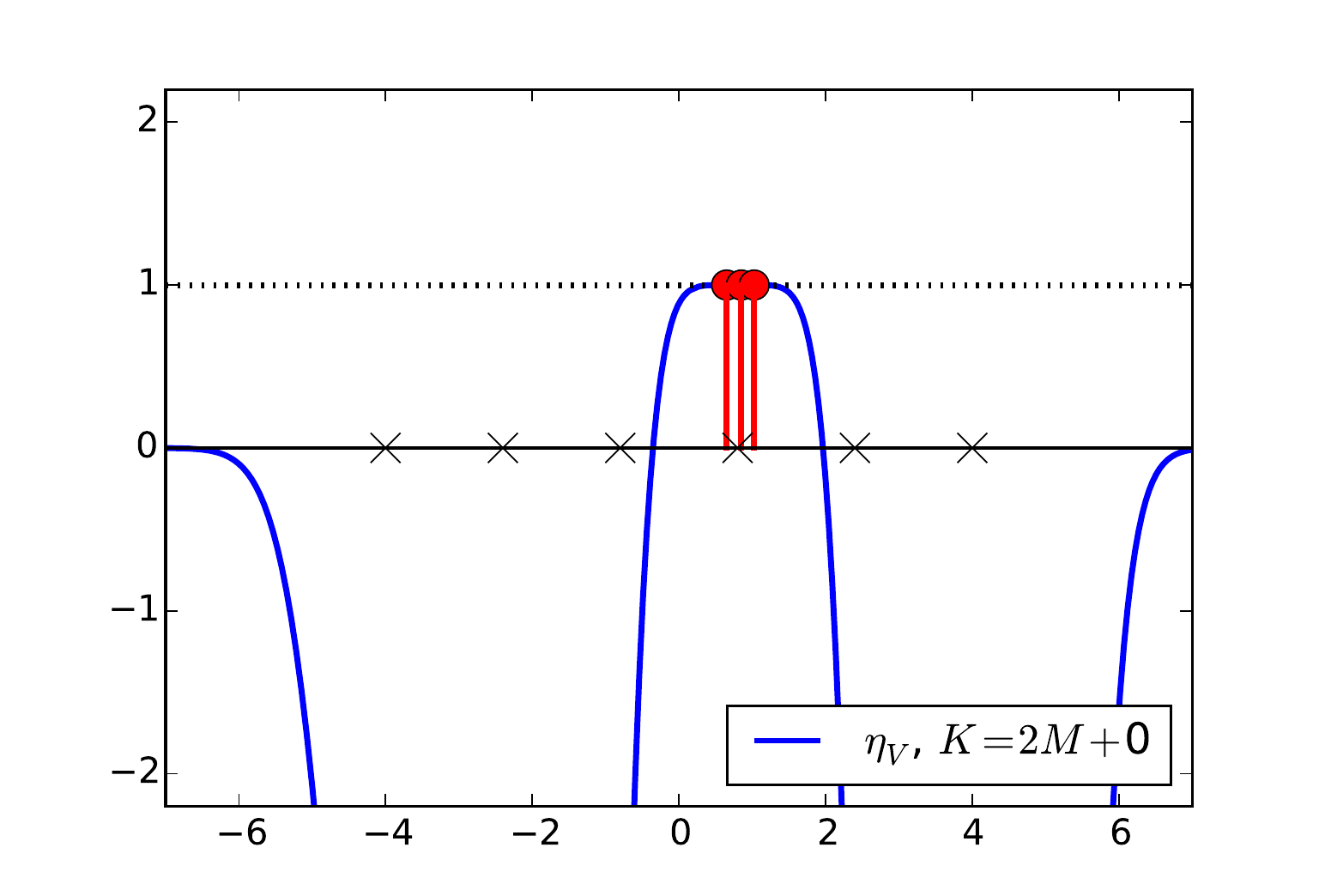}} 
  {\includegraphics[width=0.32\linewidth,clip,trim=1cm 0.5cm 1cm 0.9cm]{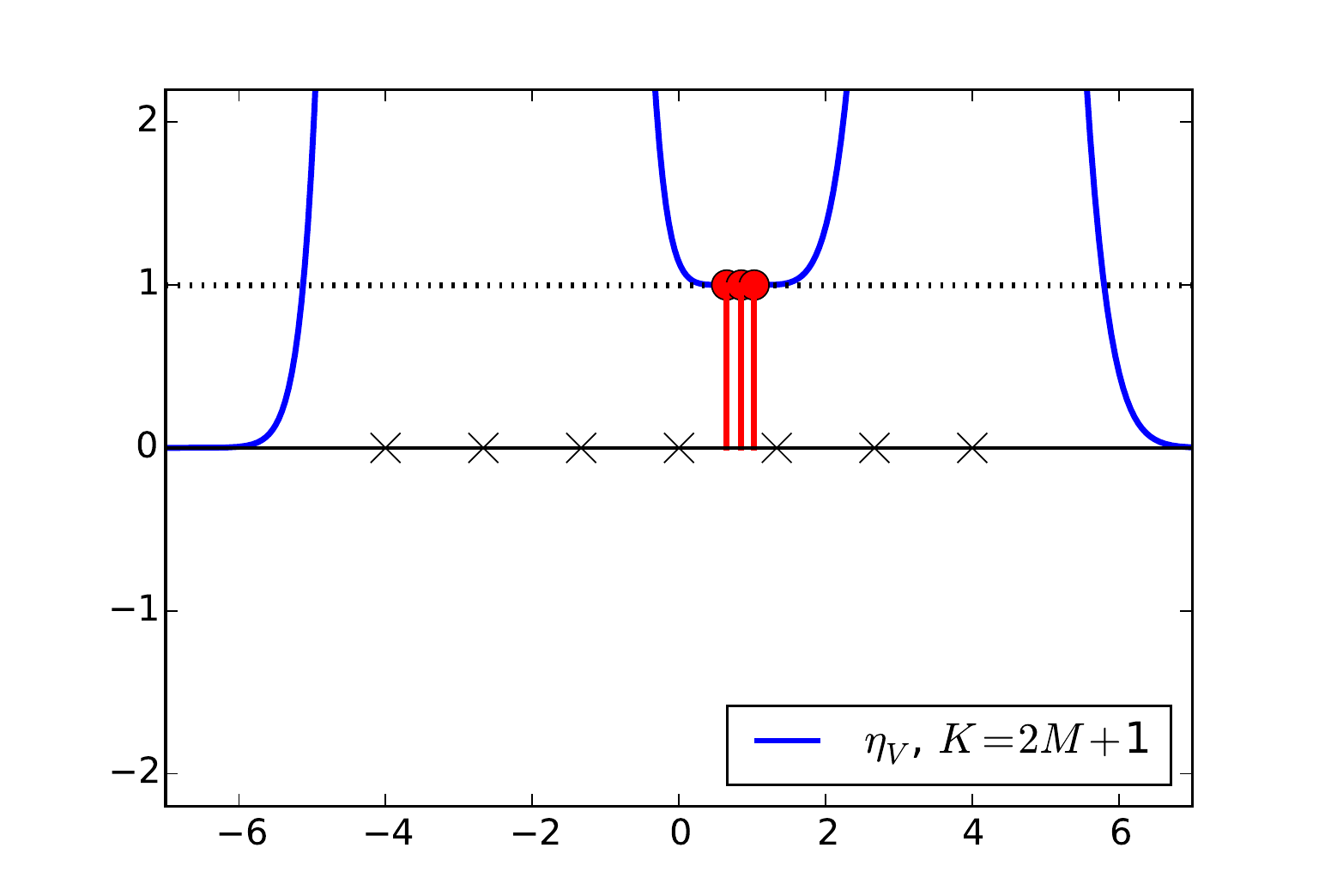}} 
  {\includegraphics[width=0.32\linewidth,clip,trim=1cm 0.5cm 1cm 0.9cm]{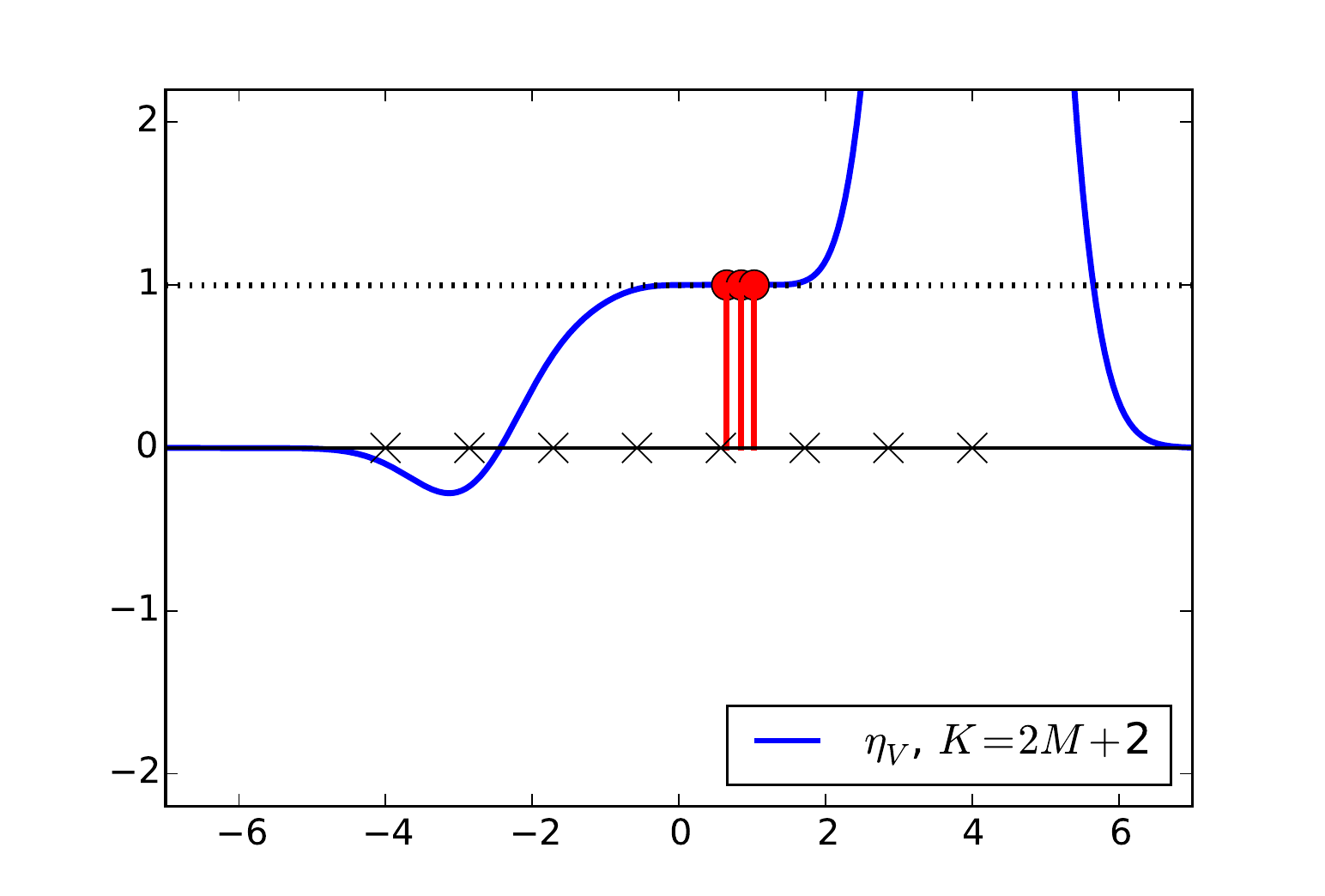}}\\ 
  {\includegraphics[width=0.32\linewidth,clip,trim=1cm 0.5cm 1cm 0.9cm]{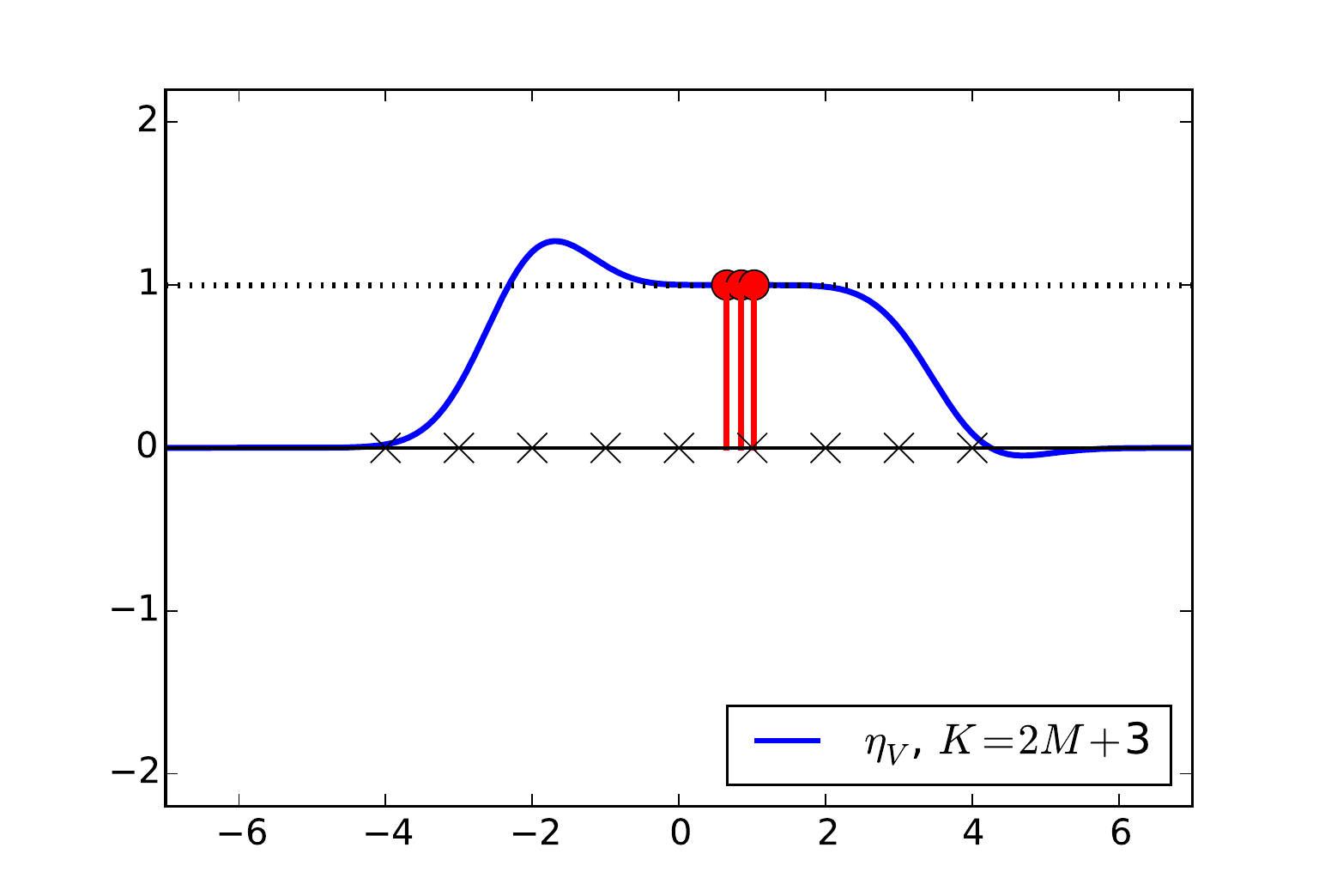}} 
  {\includegraphics[width=0.32\linewidth,clip,trim=1cm 0.5cm 1cm 0.9cm]{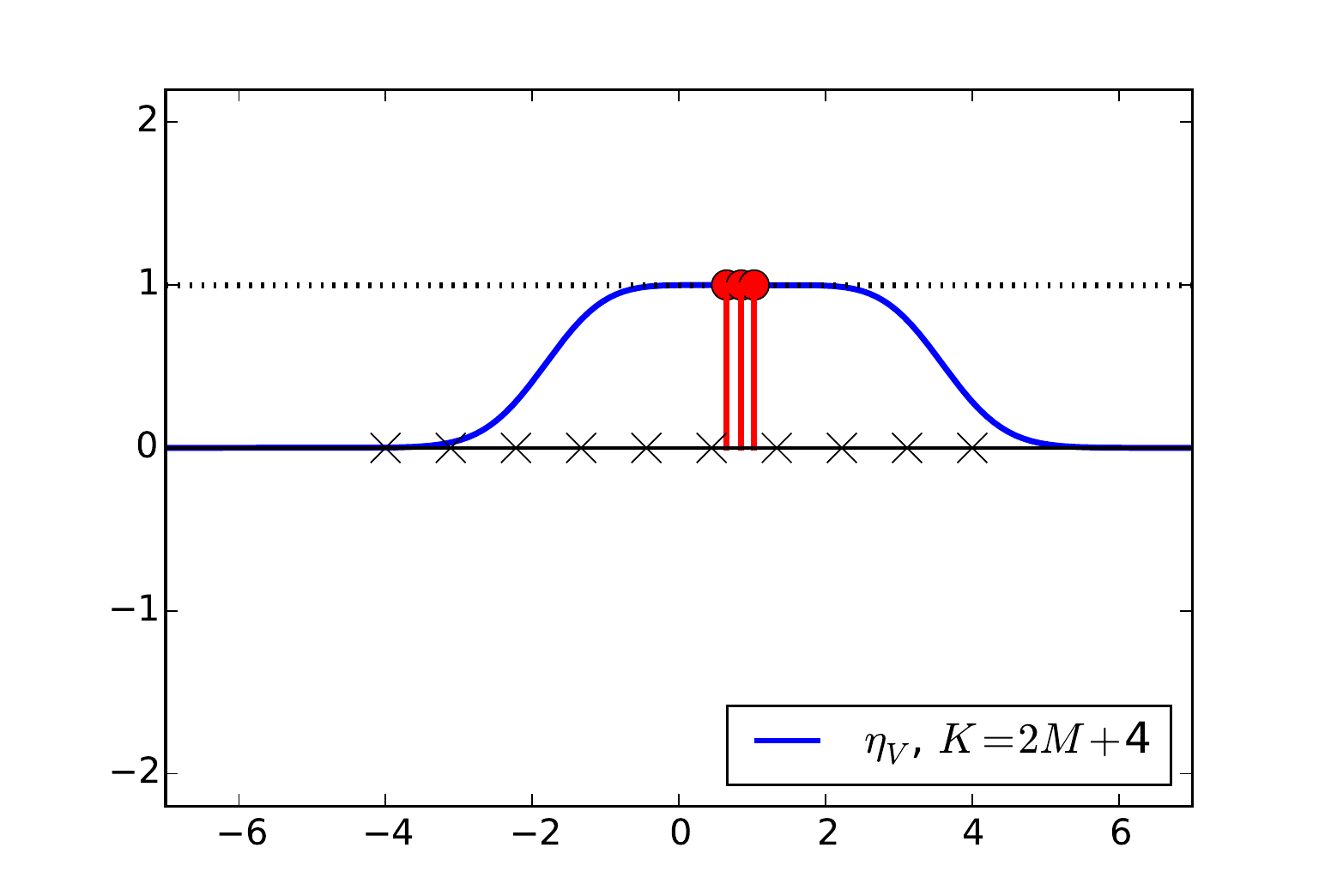}} 
  {\includegraphics[width=0.32\linewidth,clip,trim=1cm 0.5cm 1cm 0.9cm]{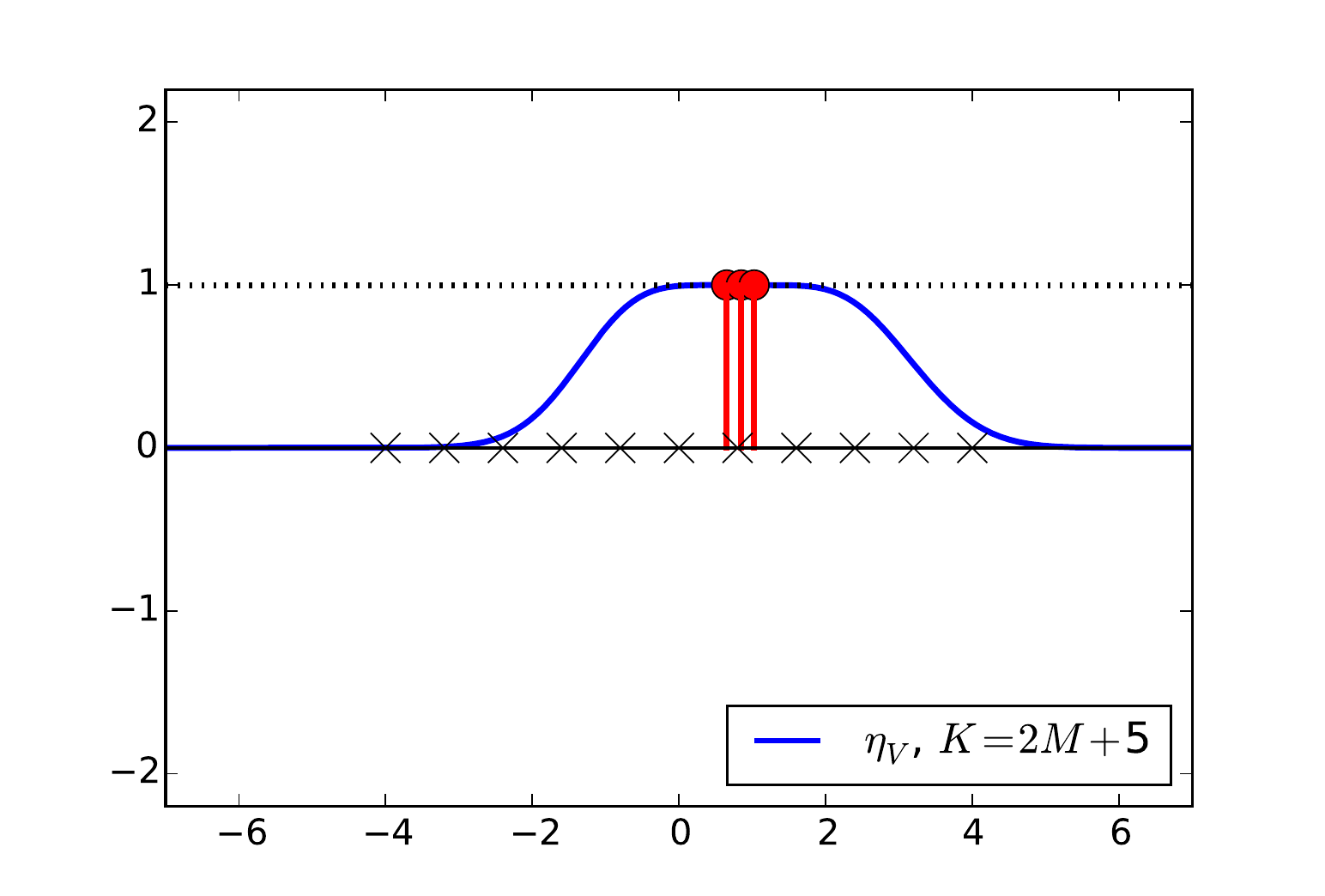}} 
  \caption{\label{fig:gauss}Vanishing derivatives precertificate for the sampled convolution with a Gaussian kernel $e^{-(\pos-\s)^2}$ (the locations of the samples $s$ are indicated by black crosses). The stability criterion ($\etaV<1$) is not always satisfied, but if the sampling measure $\Ps$ approximates the Lebesgue measure sufficiently well, it does hold.}
\end{figure}

\begin{figure}
	\centering  
  {\includegraphics[width=0.32\linewidth,clip,trim=1cm 0.5cm 1cm 0.9cm]{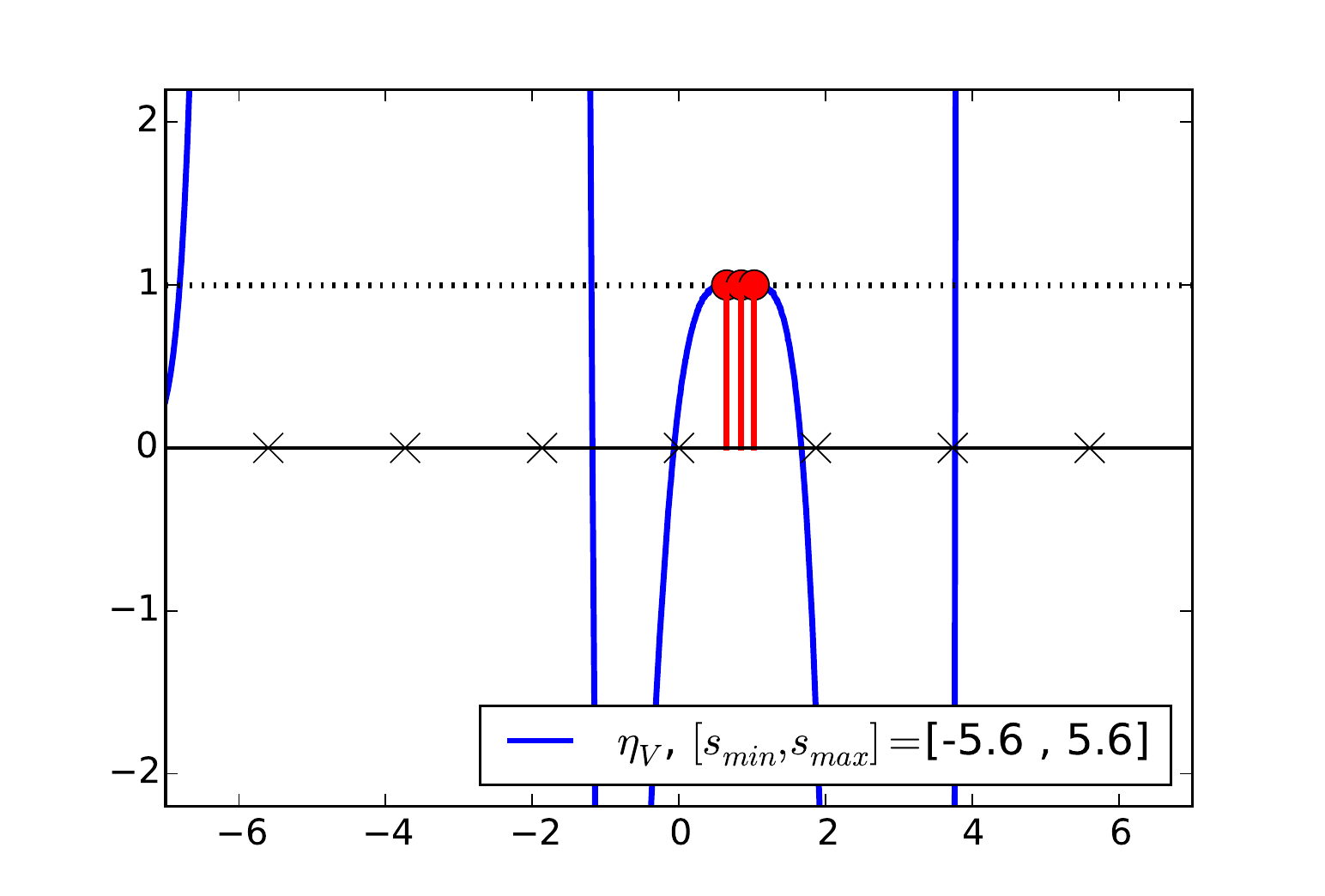}} 
  {\includegraphics[width=0.32\linewidth,clip,trim=1cm 0.5cm 1cm 0.9cm]{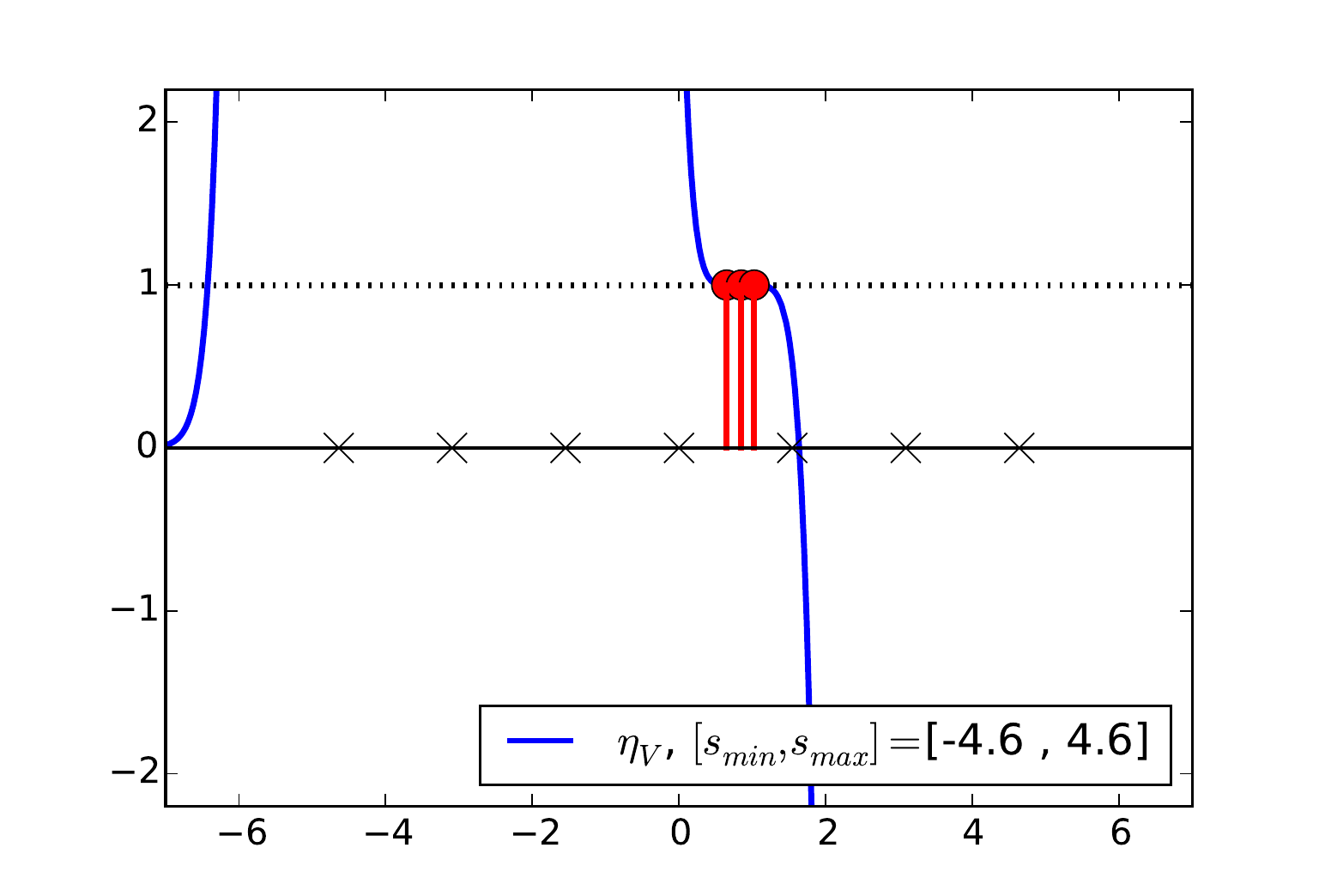}} 
  {\includegraphics[width=0.32\linewidth,clip,trim=1cm 0.5cm 1cm 0.9cm]{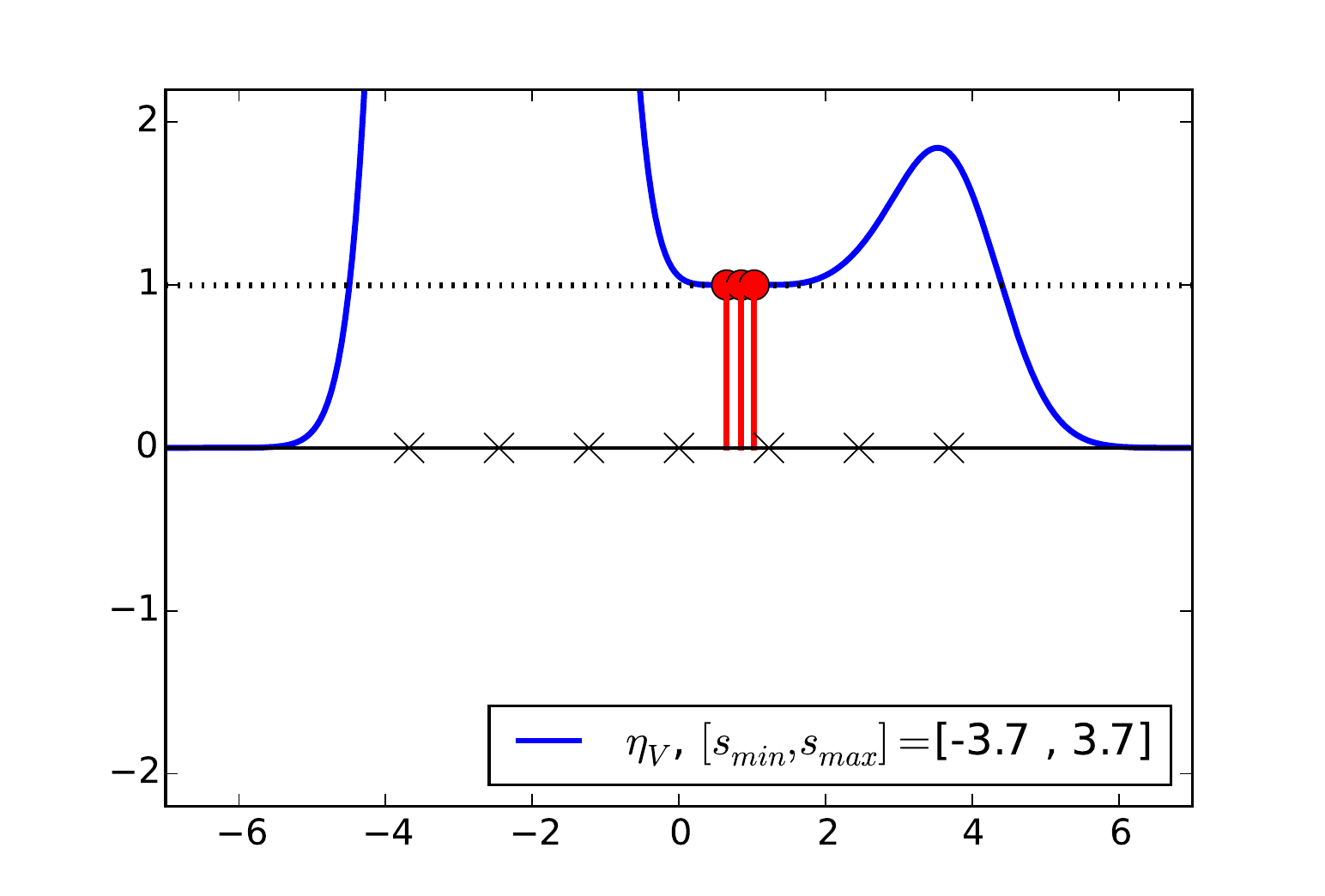}}\\ 
  {\includegraphics[width=0.32\linewidth,clip,trim=1cm 0.5cm 1cm 0.9cm]{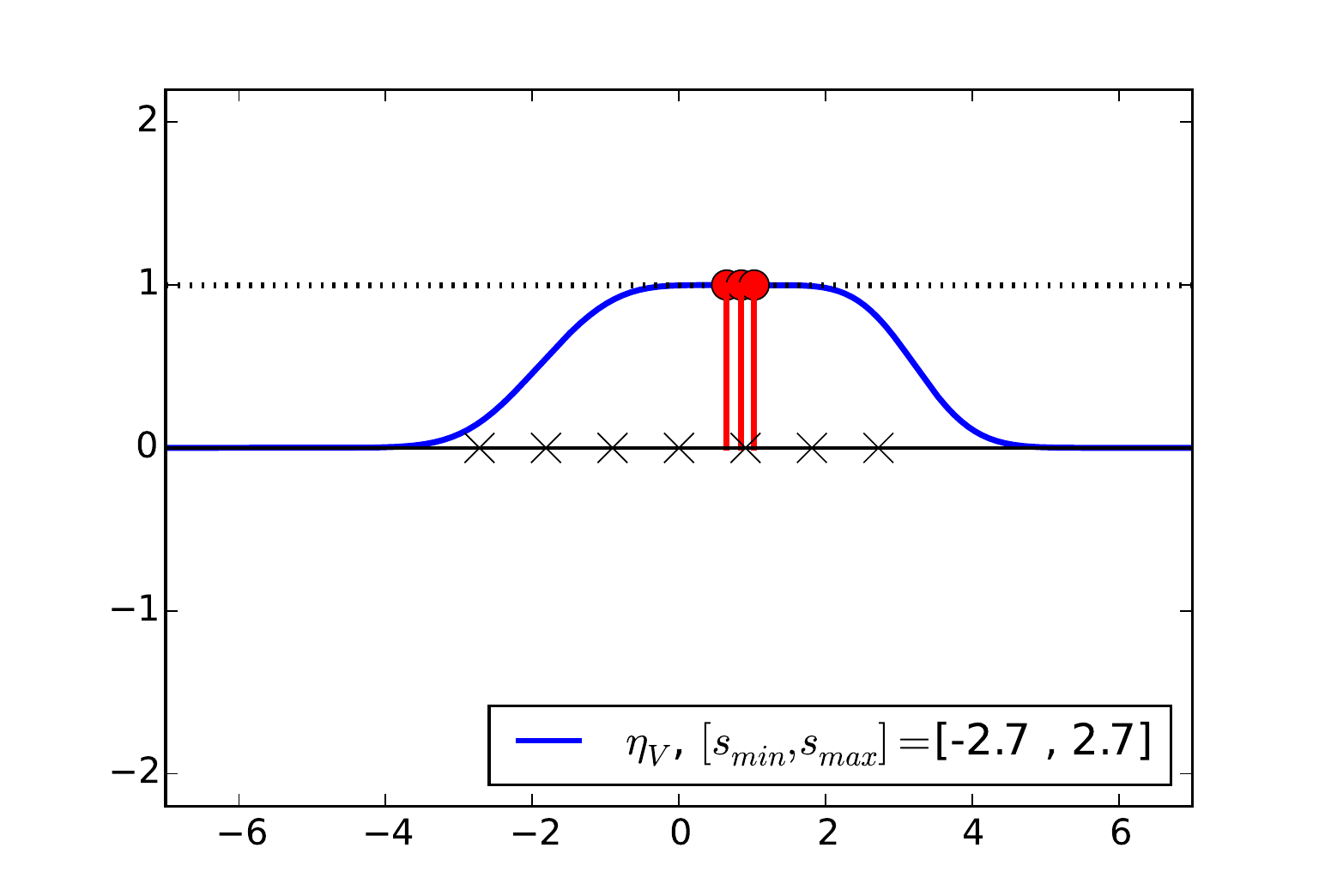}} 
  {\includegraphics[width=0.32\linewidth,clip,trim=1cm 0.5cm 1cm 0.9cm]{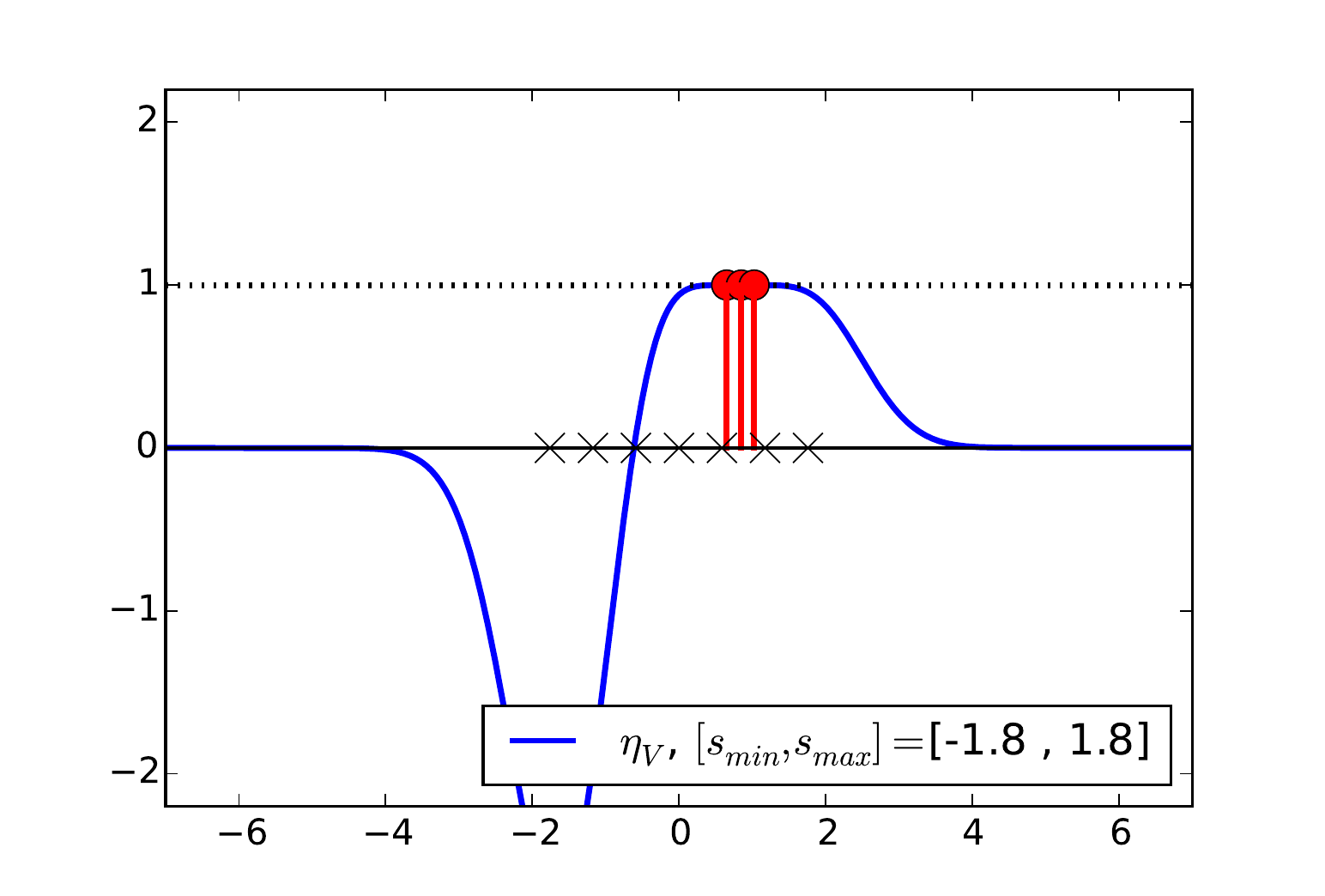}} 
  {\includegraphics[width=0.32\linewidth,clip,trim=1cm 0.5cm 1cm 0.9cm]{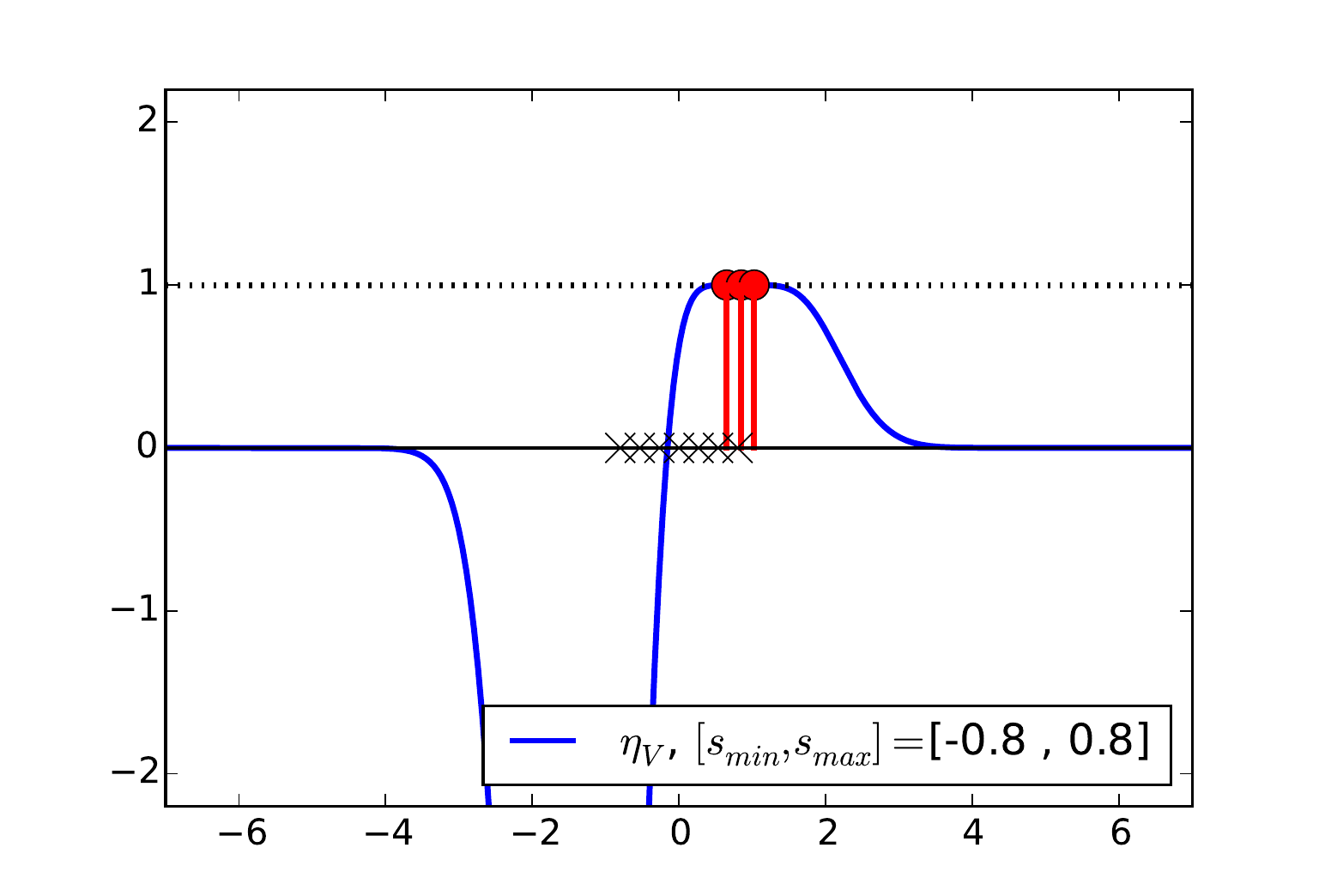}} 
  \caption{\label{fig:confinedgauss}Vanishing derivatives precertificate for the sampled convolution with a Gaussian kernel $e^{-(\pos-\s)^2}$ (the locations of the samples $s$ are indicated by black crosses). If the samples are sufficiently confined (bottom row), the precertificate is valid, $\etaV< 1$.}
\end{figure}

\subsection{Previous Works}
\paragraph{Discrete $\ell^1$ regularization} The use of $\ell^1$-type techniques for the resolution of inverse problems originates from works in geophysics for seismic exploration~\cite{Claerbout-geophysics,santosa1986linear,Levy-Fullagar-81}. Their popularization in signal processing and statistics comes respectively from the seminal works of  Donoho~\cite{donoho-superresolution1992,chen-atomic1998} and Tibshirani~\cite{tibshirani-regression1994}. One usually defines a discrete grid and solves a variational problem on that grid. The literature concerning the robustness of such models is abundant, let us mention~\cite{DossalMallat,fuchs2004on-sp,tropp2006just} which guarantees for exact support recovery, and~\cite{Grasmair-cpam} which bounds the $L^2$ recovery error.

\paragraph{Off-the-grid methods} However, using a discrete grid only approximates the problem to be solved~\cite{tang2013sparse,PiaThesis}, generally breaking the sparsity assumption of the signal (the so-called \textit{basis mismatch} phenomenon~\cite{chi2011sensitivity}): some resulting artifacts are described in~\cite{2017-Duval-IP}.
To circumvent the problem, several researchers~\cite{deCastro-exact2012,bhaskar-atomic2011,bredies-inverse2013,candes-towards2013} proposed to work in a gridless setting, turning the finite dimensional \textit{basis pursuit} or \textit{lasso} problems into variational problems in the space of Radon measures. Within that new framework, the notion of minimum separation of the spikes  is the cornerstone of several identifiability results~\cite{deCastro-exact2012,candes-towards2013,tang2013compressed} as well as robustness results~\cite{candes-superresolution2013,azais-spike2014,2015-duval-focm}. 
The particular case of positive sources, which does not impose any separation of the spikes, has been studied in~\cite{deCastro-exact2012,schiebinger2015superresolutionjournal} for identifiability, and in~\cite{2017-denoyelle-jafa} for noise robustness (see also~\cite{candes-stable2014} in a discrete setup). In dimension $d\geq 2$, the situation is considerably more difficult, the staility constants are different, as shown in~\cite{poon_multi-dimensional_2017}.

\paragraph{Non variational methods} Let us also note that, while we focus here on a variant of the Lasso for super-resolution, there is a large panel of methods designed to tackle super-resolution problems. To name a few, MUSIC~\cite{Odendaal-music} ESPRIT~\cite{Kailath_1990}, finite rate-of-innovation~\cite{vetterli-sparse2008}, matrix pencil~\cite{hua1990matrix} or Cadzow denoising~\cite{Condat-Cadzow} techniques are also worth considering. Let us also mention \cite{demanet-recoverability2014,traonmilin:hal-01509863} which tackle the super-resolution problem using the (non-convex) tools from sparsity ($\ell^0$ norm or $k$-sparse vectors). Note that there is a deep connection between super-resolution and machine learning problems such as the recovery of probability densities from random empirical generalized moments~\cite{gribonval:hal-01544609}, which paves the way for efficient methods in the setting of random measurements.
It seems that, for a given problem, the method of choice will strongly depend on the desired properties: flexibility on the measurement operator, robustness to noise, ability to handle signals with signs or phases. We do not claim that the Lasso for measures is universally better than the above-mentioned methods, we simply regard~\eqref{eq:intropostvmin} as an interesting variational problem that we wish to understand.

\subsection{Contributions}
In the present work, we propose a necessary and sufficient condition for the \textit{Non-Degenerate Source Condition} (Theorem~\ref{thm:tsystem}) which ensures support stability when trying to solve the super-resolution problem using~\eqref{eq:lassoposintro}. This condition can be seen as some infinitesimal variant of the Determinantal condition in~\cite{schiebinger2015superresolutionjournal}, its main advantage is that it ensures the non-vanishing of the second derivatives.

We show that when the observations consist of several samples of an integral transform, it is possible to ensure a priori, for some kernels, that those conditions hold regardless of the positions of the (sufficiently many) samples. Such is the case of the partial Laplace transform or the sampled convolution with a Gaussian (with an $L^1$ normalization).
In the case of the unnormalized Gaussian, we show that similar results hold provided the samples approximate the uniform Lebesgue measure sufficiently well, or, on the contrary, if they are drawn in a sufficiently small interval.

The proposed results heavily rely on the techniques introduced in~\cite{schiebinger2015superresolutionjournal} in the field of sparse recovery. Nevertheless, we develop here a self-contained theory, which benefits from the following key differences:

 \begin{itemize}
  \item Our approach is free of normalization, whereas~\cite{schiebinger2015superresolutionjournal} requires an $L^1$-type normalization of the atoms. Despite the good empirical results reported in~\cite{schiebinger2015superresolutionjournal}, it is arguable whether this normalization is the one to use, as the problem~\eqref{blasso} seems naturally Hilbertian. Also, we clarify the relationship between the $L^1$ normalization and the invoked conditions.

  \item On a related topic (see Section~\ref{sec:cauchy}), the approach in~\cite{schiebinger2015superresolutionjournal} naturally requires $2\M+1$ observations of the filtered signal (where $m_0$ has $\M$ spikes), our approach without normalization only needs $2\M$, which is the optimal requirement.

  \item Most importantly, we ensure the stability of the support at low noise, whereas~\cite{schiebinger2015superresolutionjournal} only provides identifiability.
\end{itemize}

\subsection{Notations and Preliminaries}
\label{sec:intro-notations}
\paragraph{Functional spaces}
Let $\dompos$ be an interval of $\RR$. We denote by $\CoX$ the space of $\RR$-valued continuous functions which vanish ``at infinity'', \ie for all $\varepsilon>0$, there exists a compact subset $K\subset \dompos$ such that 
\begin{align*}
  \forall x\in \dompos\setminus K,\quad \abs{\eta(x)}&\leq \varepsilon.
\end{align*}
More explicitly, \begin{itemize}
  \item If $\dompos=[a,b]\subset \RR$, $\CoX=\Cder{}([a,b])$ is the space of continuous functions on $[a,b]$.
  \item If $\dompos=[a,b)\subset \RR$, $\CoX$ is the space of continuous functions $\eta$ on $[a,b)$ such that $\lim_{\pos\to b}\eta(\pos)=0$.
  \item If $\dompos=(a,b)\subseteq \RR$, $\CoX$ is the space of continuous functions $\eta$ on $(a,b)$ such that $\lim_{\pos\to a}\eta(\pos)=\lim_{\pos\to b}\eta(\pos)=0$.
\end{itemize}

Endowed with the supremum norm, $\CoX$ is a Banach space whose dual is the space $\radon$ of finite Radon measures (see~\cite{AFP00}). We denote by $\radonpos$ the set of nonnegative Radon measures. In the rest of the paper, we regard $\CoX$ and $\radon$ as locally convex topological vector spaces, and we endow them respectively with their weak and weak-* topologies. With that choice, both spaces are dual to each other.

%
The functional 
\begin{align*}
\lmass(m)&\eqdef  m(\dompos)+\chi_{\radonpos}=\begin{cases}
  m(\dompos) & \mbox{if $m$ is nonnegative,}\\
  +\infty & \mbox{otherwise.}
  \end{cases}
\end{align*}
can be written as a support function 
\begin{align*}
  \lmass(m) &=\sup\enscond{\int_{\dompos}\eta\d m}{\eta\in \CoX \mbox{ and }\sup_{\pos\in \dompos}{\eta}(\pos)\leq 1},
 \end{align*}
so that its subdifferential is
\begin{align}\label{eq:subdiffmass}
  \partial\lmass(m)= \enscond{\eta\in \CoX}{\sup{\eta}\leq 1 \ \mbox{and}\ \int_\dompos \eta\d m=m(\dompos) }.
\end{align}

\paragraph{Linear operators}
Let $\Hh$ be a real separable Hilbert space. Let $\varphi \in\Cder{}(\dompos;\Hh)$ be a continuous $\Hh$-valued function such that for all $p\in \Hh$, $\left(x\mapsto \dotpH{p}{\varphi(x)}\right)\in \CoX$. Note that by the Banach-Steinhaus theorem, $\sup_{\pos\in\dompos}\normH{\varphi(x)}<+\infty$.

Then, the linear operator $\Phi:\radon\rightarrow \Hh$,
\begin{align}\label{eq:defPhi}
  \Phi m\eqdef \int_{\dompos}\varphi \d m  
\end{align}
(where the above quantity is a Bochner integral~\cite[Sec. V.5]{yosida1995functional})
is weak-* to weak continuous, as shown by the equality
\begin{align*}
  \dotpH{\Phi m}{p}= \dotpH{\int_{\dompos}\varphi \d m }{p}=\int_{\dompos}\dotpH{\varphi(x)}{p} \d m(x).
\end{align*}
It also shows that the adjoint $\Phi^*: \Hh\rightarrow \CoX$ has full domain, with $(\Phi^*p): x\mapsto \dotpH{\varphi(x)}{p}$. We note that $\Phi^*$ is continuous (from the strong to the strong topology as well as) from the weak to the weak topology.

\paragraph{Admissible Kernels}
We say $\varphi$ is in $\kernel{k}$ if $\varphi \in\Cder{k}(\dompos;\Hh)$ and 
\begin{itemize}
  \item for all $p\in \Hh$, $\left(x\mapsto \dotpH{p}{\varphi(x)}\right)\in \CoX$.
  \item for all $0\leq i\leq k$, $\sup_{x\in\dompos}\normH{\varphi^{(i)}(x)}<+\infty$,
\end{itemize}

Given $\bpos=(\pos_1,\ldots,\pos_\M)\in\dompos^\M$, we denote by $\Phi_{\bpos}:\RR^\M\rightarrow \Hh$ the linear operator such that 
\eq{
	\forall a\in \RR^\M, \quad	
	\Phi_{\bpos}(a) \eqdef \sum_{i=1}^{\M} a_i \phi(\pos_i),
}
and by $\Ga_{\bpos}:\RR^{2\M}\rightarrow \Hh$ the linear operator defined by
\eql{\label{eq:defgamma}
	\Ga_{\bpos}b \eqdef 
  \sum_{i=1}^{\M}\left( b_{2i-1} \phi(\pos_i) + b_{2i} \varphi'(\pos_i)\right).
}
We adopt the following matricial notation 
\eq{
	\Phi_{\bpos}=(\phi(\pos_1) \ \ldots \ \phi(\pos_{\M})) \qandq \Ga_{\bpos}=(\phi(\pos_1),\phi'(\pos_1) \ \ldots \ \phi(\pos_{\M}),\phi'(\pos_{\M})),
}
where the ``columns'' of those ``matrices'' are elements of $\Hh$. In particular, the adjoint operator $\Phi_{\bpos}^*:\Hh\rightarrow \RR^\M$ is given by
\eq{
  \forall p\in\Hh,\quad  \Phi_{\bpos}^*p=\left((\Phi^*p)(\pos_i)\right)_{1\leq i\leq \M}.
}

Given $\pos_0\in \dompos$, we denote by $\phiD{k} \in \Hh$ the $k^{th}$ derivative of $\phi$ at $\pos_0$, \textit{i.e.}
\eql{\label{eq-defn-phider}
\phiD{k} \eqdef \phi^{(k)}(\pos_0).
}
In particular, $\phiD{0} = \phi(\pos_0)$.

Given $k \in \NN$, we define
\eql{\label{eq-defn-Fk}
	\Fk \eqdef \begin{pmatrix} \phiD{0} & \phiD{1} & \ldots & \phiD{k} \end{pmatrix}.
}
If $\Fk: \RR^{k+1}\rightarrow \Hh$ has full column rank, we define its pseudo-inverse as $\Fk^+\eqdef (\Fk^*\Fk)^{-1}\Fk^*$. Similarly, we denote $\Gamma_{\bpos}^+ \eqdef (\Gamma_{\bpos}^*\Gamma_{\bpos})^{-1}\Gamma_{\bpos}^*$  provided $\Gamma_{\bpos}$ has full column rank.


\section{The \blassosc{} with nonnegativity constraints}
\label{sec:blasso}

In this section, we describe the main properties of the positive Beurling LASSO (\pblassosc) that are needed in the paper. These are mainly minor adaptations of the properties stated in~\cite{2015-duval-focm,2017-denoyelle-jafa} for the \blassosc, and we state them here without proof. 

We assume that we are given a Hilbert space $\Hh$ and some kernel $\varphi\in \kernel{k}$ for some $k\geq 2$ as described above.
The choice of $\varphi$ and $\Hh$ is discussed in detail in Section~\ref{sec:blassoreconstruction}

\subsection{Variational Problems in the Space of Measures}

Our goal is to recover an unknown input measure 
\eq{
	\mo = \m_{\amp,\bpos} \eqdef \sum_{i=1}^\M \amp_{i} \dirac_{\pos_{i}}
} 
with $(\amp_{i},\pos_{i}) \in \RR \times \dompos$  from the observations $y=\Phi\mo+w\in \Hh$, where $\Phi$ is the linear operator defined in~\eqref{eq:defPhi}, and $w\in \Hh$ is some additive noise. In this paper, we assume that the amplitudes $\amp_{i}$ are positive.

To perform the reconstruction, we solve the following variant of the \blassosc{}, adding a nonnegativity constraint, that we denote by \pblassosc,
\begin{align}\label{blasso}\tag{$\blasso{y}$}
	\underset{m\in\radonpos}{\min}\ \la \m(\dompos)+ \frac{1}{2} \normH{\Phi \m - y}^2,
\end{align}
where $\la>0$ is some regularization parameter.
In the noiseless case ($w=0$), one might as well solve the constrained problem
\begin{align}\label{eq:bpursuit}\tag{$\bpursuit{\Phi\mo}$}
  \underset{\substack{m\in\radonpos\\ \Phi m=\Phi\mo}}{\min}\ \m(\dompos).
\end{align}
It worth noting that~\eqref{eq:bpursuit} is the limit of~\eqref{blasso} as $\la\to 0$ and $\normH{w}\to 0$ in a suitable way (see~\cite{2015-duval-focm} for the case of the \blassosc).

\subsection{Low noise behavior and exact support recovery}
From now on, in order to simplify the discussion, we assume that $\pos_i\in \interop(\dompos)$ for all $1\leq i\leq \M$, \ie $\inf \dompos<\pos_{i}<\sup\dompos$.

\paragraph{Exact support recovery}
In~\cite{2015-duval-focm}, we have provided an almost sharp sufficient condition for the exact support recovery at low noise.
A (simplified) version of that condition is the following.
If $\Gamma_{\bpos}$ (see~\eqref{eq:defgamma}) has full rank, we define the vanishing derivatives precertificate as
\begin{align}\label{eq:defpv}
  \etaV&\eqdef \Phi^*\pV, \qwhereq\\
  \pV&\eqdef\argmin\enscond{\normH{p}}{p\in\Hh, (\Phi^*p)(\pos_i)=1, (\Phi^*p)'(\pos_i)=0\ \mbox{for all $1\leq i\leq\M$} }.
\end{align}
In other words, $\pV=\Gamma_{\pos}^{+,*}u$ where $u\eqdef (1,0,1,0\ldots, 1,0)^T\in \RR^{2\M}$, and $\Gamma_{\pos}^{+,*}=\Gamma_{\pos}(\Gamma_{\pos}^*\Gamma_{\pos})^{-1}$ is the Moore-Penrose pseudo-inverse of $\Gamma_{\pos}$.
We note that by optimality of $\pV$ in~\eqref{eq:defpv}, there exists some coefficients $\{\alpha_i,\beta_i\}_{1\leq i\leq \M}\subset \RR$ such that 
    $\pV = \sum_{i=1}^\M \left(\alpha_i\varphi(\pos_i) + \beta_i\varphi'(\pos_i)\right)$, and $\etaV$ is the unique function of the form
    \begin{align}
      \etaV(\pos)&=\sum_{i=1}^\M \left(\alpha_i\dotpH{\varphi(\pos)}{\varphi(\pos_i)} + \beta_i\dotpH{\varphi'(\pos)}{\varphi'(\pos_i)}\right) \label{eq:etaVform}
    \end{align}
  which satisfies $\etaV(\pos_i)=1$ and $\etaV'(\pos_i)=0$ for all $i\in\{1,\ldots,\M\}$.

\begin{definition}[NDSC]
Assume that $\Gamma_{\bpos}$ has full rank. We say that  $\mo$ satisfies the \textit{Non Degenerate Source Condition} if $\etaV(\pos)<1$ for all $\pos\in \dompos\setminus\{\pos_1,\ldots,\pos_\M\}$ and $\etaV''(\pos_i)<0$ for all $1\leq i\leq \M$.
\end{definition}

As the next result shows, the nondegeneracy of $\etaV$ almost characterizes the support stability of the \pblassosc at low noise. When it holds, it entails $\abs{\tilde{\amp}_{i}(\la,w)-\amp_{i}}=O(\normH{w})$ and $\abs{\tilde{\pos}_i(\la,w)-\pos_i}=O(\normH{w})$ for $\la=\frac{1}{\alpha}\normH{w}$.

\begin{theorem}[\cite{2015-duval-focm}]\label{thm-duval}
  Assume that $\Gamma_{\bpos}$ has full rank and that $\mo$ satisfies the Non Degenerate Source Condition. Then there exists $\alpha>0$, $\lambda_0>0$ such that for all $0\leq \la\leq \la_0$ and $\normH{w}\leq \alpha\la$, the solution $\tilde{m}_{\la,w}$ to~\eqref{blasso} is unique and is composed of exactly $\M$ spikes, with $\tilde{m}_{0,0}=\mo$.

  Moreover one may write $\tilde{m}_{\la,w}=\sum_{i=0}^{\M}\tilde{\amp}_{i}(\la,w)\delta_{\tilde{\pos}_i(\la,w)}$, where $\tilde{\amp}_{i}$ and $\tilde{\pos}_i$ are the restriction of $\Cder{1}$ functions on a neighborhood of $(0,0)$.

  Conversely, if $\Gamma_{\bpos}$ has full rank and the solutions can be written in the above form with functions $\tilde{\amp}_{i}$ and $\tilde{\pos}_i$ which are $\Cder{1}$ and satisfy $\tilde{\amp}_{i}(0,0)=\amp_{i}$, $\tilde{\pos}_i(0,0)=\pos_i$, then $\etaV\leq 1$.
\end{theorem}

Unfortunately, the Non Degenerate Source Condition (NDSC) is not trivial to ensure a priori. One reason is that it should be checked for any configuration $\pos_1<\ldots <\pos_\M$. Let us mention~\cite{tang2013atomic} which proves that the (NDSC) holds for the Gaussian or Cauchy kernels and signed measures, provided the spikes are sufficiently separated. As noted above in the introduction, in the signed case, the (NDSC) generally cannot hold if spikes with opposite signs are too close (see~\cite[Corollary 1]{2015-duval-focm} and~\cite{tang2015resolution}).

The situation is more favorable in the setting of nonnegative measures that we consider. In Section~\ref{sec:tsystem}, building upon~\cite{schiebinger2015superresolutionjournal}, we show that the (NDSC) holds for specific choices of $\phi$ which include the Laplace kernel. 

\paragraph{The case of clustered spikes}
Another approach for nonnegative measures, taken in~\cite{2017-denoyelle-jafa}, is to consider spikes which are located in a small neighborhood of a single point $\pos_0\in \interop(\dompos)$, say $\pos_i= \pos_0 + tz_{i}$, for some $z_{i}\in \RR$ and some small $t>0$.
Let us write $m_t=\sum_{i=1}^\M a_i\delta_{x_0+tz_i}$ and $y_t=\Phi m_t+w$. One may prove that, as $t\to 0^+$, the corresponding $\pV$ converges (strongly in $\Hh$) towards 
\begin{align}
  \label{eq:defpW}
 \pW=\argmin\enscond{\normH{\pW}  }{(\Phi^*\pW)(x_0)=1, (\Phi^*\pW)'(x_0)=\ldots =(\Phi^*\pW)^{(2\M-1)}(x_0)=0}
\end{align}
provided $\varphi\in\kernel{2\M-1}$ and $(\phiD{0},\ldots,\phiD{2\M-1})$ has full rank.
As a result $\etaV=\Phi^*\pV$ converges (uniformly on $\dompos$) towards $\etaW\eqdef \Phi^*\pW$, and the same holds for the derivatives up to the regularity of $\varphi$. 
Letting $\Fdn\eqdef(\phiD{0},\ldots,\phiD{2\M-1})$, we note as above that $\pW\eqdef\Fdn^{+,*}v$ where $v\eqdef (1,0,\ldots,0)^T\in \RR^{2\M}$.
Moreover, $\etaW$ is the unique the function of the form 
\begin{equation*}
  \etaW: x\mapsto \sum_{k=0}^{2\M-1}\beta_k\dotpH{\varphi(x)}{\varphi^{(k)}(x_0)},
\end{equation*}
which satisfies 
\begin{equation}
  \label{eq:vanishing}
  \etaW(x_0)=1,\ \etaW'(x_0)=0,\ \ldots,\ \etaW^{(2\M-1)}(x_0)=0.
\end{equation}

\begin{definition}[$(2\M-1)$-Nondegeneracy]
We say that a point $\x_0$ is \textit{$(2\M-1)$ Non Degenerate} if 
\begin{itemize}
  \item $\varphi\in\kernel{2\M}$,
  \item $\Fdn\eqdef(\phiD{0},\ldots,\phiD{2\M-1})$ has full rank,
  \item $\etaW^{(2\M)}(\x_0)<0$,
  \item for all $\pos\in \dompos\setminus\{\pos_0\}$, $\etaW(\pos_0)<1$.
\end{itemize}
\end{definition}

The $(2\M-1)$-Nondegeneracy of a point entails the support stability of finite measures which are clustered around $\x_0$.
\begin{theorem}[\cite{2017-denoyelle-jafa}]\label{thm-denoyelle}
  Assume that $\varphi\in \kernel{2\M+1}$ and that the point is $(2\M-1)$ Non Degenerate. Then there exists $t_0>0$, such that  $\m_t$ satisfies the (NDSC) for for all $0<t<t_0$.
  
  More precisely, there exists $\alpha>0$, $C_R>0$ and $C>0$ such that for $0<t<t_0$ and all $0<\la+\normH{w}<C_Rt^{2\M-1}$ with $\normH{w}\leq \alpha\la$,
\begin{itemize}
  \item the problem $\blasso{y_t}$ has a unique solution of the form $\sum_{i=1}^{\M}\tilde{\amp_i}(\la,w,t)\delta_{x_0+t\tilde{z}_i(\la,w,t)}$,
  \item the mapping $g_t:(\la,w)\mapsto (\tilde{\amp_i},\tilde{z}_i)$ is the restriction of a $\Cder{2\M}$ function,
  \item the following inequality holds
\begin{align*}
  \abs{(\tilde{\amp},\tilde{z})-(\amp,z)}_\infty\leq C\left(\frac{\abs{\la}}{t^{2\M-1}} +\frac{\normH{w}}{t^{2\M-1}} \right).
\end{align*}
\end{itemize}
\end{theorem}
 Theorem~\ref{thm-denoyelle} is more quantitative than Theorem~\ref{thm-duval}, as it gives the scaling of the constants with respect to the spacing $t$ of the spikes, describing how it becomes ill-conditioned.

\subsection{Ensuring negative second derivatives}
In general, the conditions $\etaV''(\pos_i)<0$ and $\etaW^{(2\M)}(\x_0)<0$ in the non degeneracy hypotheses are difficult to ensure a priori. They can be checked numerically, or, in the case of $\etaW$, analytically on a closed form expression. But it is sometimes possible to ensure them a priori: in~\cite[Proposition 5]{2017-denoyelle-jafa}, we have proved that $\etaW^{(2\M)}(\x_0)<0$ for fully sampled convolutions on translation-invariant domains, under a mild assumption. 

In the present paper we propose the following two new Lemmas, which work for any sampling of the impulse response. They follow directly from~\eqref{eq:etaVform} and the characterization of $\etaV$, in the spirit of Cramer's rule.

\begin{lemma}\label{lem:etaVcramer}
  Assume that $\Gamma_{\bpos}$ has full rank. Then for all $i\in\{1,\ldots,\M\}$,
  \begin{align}
    \label{eq:etaVcramer}
    \begin{mydet}
    \dotpH{\phi(\pos_1)}{\phi(\pos_1)} & \dotpH{\phi(\pos_1)}{\phi'(\pos_1)} & \cdots &  \dotpH{\phi(\pos_1)}{\phi'(\pos_\M)} & 1\\
    \dotpH{\phi'(\pos_1)}{\phi(\pos_1)} & \dotpH{\phi'(\pos_1)}{\phi'(\pos_1)} & \cdots &  \dotpH{\phi'(\pos_1)}{\phi'(\pos_\M)} & 0\\
    \vdots & \vdots &  &\vdots & \vdots\\
    \dotpH{\phi(\pos_\M)}{\phi(\pos_1)} & \dotpH{\phi(\pos_\M)}{\phi'(\pos_1)} & \cdots &  \dotpH{\phi(\pos_\M)}{\phi'(\pos_\M)} & 1\\
    \dotpH{\phi'(\pos_\M)}{\phi(\pos_1)} & \dotpH{\phi'(\pos_\M)}{\phi'(\pos_1)} & \cdots &  \dotpH{\phi'(\pos_\M)}{\phi'(\pos_\M)} & 0\\
      \dotpH{\phi''(\pos_{i})}{\phi(\pos_1)} & \dotpH{\phi''(\pos_{i})}{\phi'(\pos_1)} & \cdots &  \dotpH{\phi''(\pos_{i})}{\phi'(\pos_\M)} & 0
    \end{mydet}= -\etaV''(\pos_{i})\det(\Gamma_{\bpos}^*\Gamma_{\bpos}).
  \end{align}
\end{lemma}
Ensuring that the left hand side is positive provides the condition $\etaV''(\pos_\M)<0$, and a similar argument can be derived for the other $\pos_i$'s.

The analogous property for $\etaW$ is the following.
\begin{lemma}\label{lem:blassocramer}
  Assume that $\Fdn$ has full column rank. 
  
  Then,   \begin{align}\label{eq:detalterne}
    \begin{mydet}
      \dotpH{\phiD{0}}{\phiD{0}}&   \dotpH{\phiD{0}}{\phiD{1}} &\cdots &\dotpH{\phiD{0}}{\phiD{2\M-1}} & 1\\
      \dotpH{\phiD{1}}{\phiD{0}}&   \dotpH{\phiD{1}}{\phiD{1}} &\cdots &\dotpH{\phiD{1}}{\phiD{2\M-1}} & 0\\
      \vdots &\vdots & & \vdots&\vdots\\
      \dotpH{\phiD{2\M}}{\phiD{0}}& \dotpH{\phiD{2\M}}{\phiD{1}} &\cdots &\dotpH{\phiD{2\M}}{\phiD{2\M-1}} &0
    \end{mydet}&= -\etaW^{(2\M)}(\pos_0)\det(\Fdn^*\Fdn).
  \end{align}
\end{lemma}

\subsection{Reconstruction frameworks}
\label{sec:blassoreconstruction}

Now, we discuss the choice of the Hilbert space and kernel $\varphi$ . 

\begin{rem}
Let us note that~\eqref{blasso} may be written in terms of the autocorrelation function 
\eq{
	\foralls (\pos,\pos') \in \dompos^2, \quad
	\Co(\pos,\pos') \eqdef \dotpH{ \phi(x) }{ \phi(x') }.
}
Introducing the autocorrelation operator 
\begin{align*}
  \Cop \eqdef \Phi^*\Phi \colon \radon&\rightarrow \CoX\\
                               m  &\longmapsto \pa{ \int_{\dompos} \Co(\cdot,\pos') \d m(\pos') }, 
\end{align*}
we note that~\eqref{blasso} is equivalent to
\eq{
	\frac{1}{2}\normH{y}^2 + 
  \umin{m\in\radonpos}\pa{ \la m(\dompos)+ \frac{1}{2} \dotp{\Cop \m}{\m} - \dotp{\Phi^*y}{\m}}\tag{$\tilde\Pp{y}$}
}
(here $\dotp{\cdot}{\cdot}$ denotes the duality pairing between $\CoX$ and $\radon$).
As a result, everything described above (in particular expression for certificates) can be written in term of the autocorrelation function and its derivatives. Moreover, choosing $\varphi$ and $\Hh$ is equivalent (but less intuitive) to choosing the autocorrelation function $\Co$.
\end{rem}

A typical choice of Hilbert space covered by our framework is the following. Given some positive measure $\Ps$ on some interval $\domobs\subseteq \RR$, the impulse response is 
\begin{align*}
  \varphi(x): \s \mapsto \psi(\pos,\s)
\end{align*}
where $\psi$ is some kernel and $\Hh=\Ldeux(\domobs,\Ps)$. We say that $(\domobs,\Ps)$ define a reconstruction framework. 

For instance, if $\Ps=\sum_{k=1}^K \delta_{\s_k}$ and $\psi(\pos,\s)=e^{-(\pos-\s)^2}$, our observation is the sampled convolution of $m_0$ with the Gaussian kernel,
\begin{align*}
  (\Phi m_0)_k = \int_{\dompos} e^{-(x-\s_k)}\d \m_0(x),
\end{align*}
and the norm of $\Hh$ is the standard Euclidean norm in $\RR^K$.

In the following, we usually assume that $\psi$ is smooth in its first argument, and that it is possible to differentiate under the sum symbol,
\begin{align}\label{eq:derivint}
  \varphi^{(k)}(\pos)=(\partial_1)^{k} \psi(x,\cdot)\in \Ldeux(\domobs,\Ps).
\end{align}


\section{A characterization of the Non-Degenerate Source Condition}
\label{sec:tsystem}
In this section we provide a simple criterion which ensures that a measure $\mo$ is the unique solution to the basis pursuit for measures~\eqref{eq:bpursuit}, and that its recovery using~\eqref{blasso} is stable. This criterion is an "infinitesimal version" of the criterion proposed in~\cite{schiebinger2015superresolutionjournal}, except that we do not impose any normalisation of the atoms (nor any weight on the total variation).
Throughout this section we let $\{\pos_i\}_{i=1}^\M\subset \interop(\dompos)$ with $\pos_1<\ldots<\pos_\M$, and $\pos_0\in\interop(\dompos)$.

\subsection{Rescaled determinants}
The following elementary result is a specialization of a standard trick in the study of \textit{extended T-systems}~\cite{karlin1966tchebycheff}. For sufficiently regular functions $\{u_0,\ldots, u_{2\M}\}$, define $\detV$ and $\detW$ by
    \begin{align}
      \forall t\in \dompos\setminus\{\pos_i\}_{i=1}^{\M},\quad      \detV(t)&\eqdef \frac{2}{\prod_{i=1}^\M (t-\pos_i)^2} \begin{mydet}
        u_0(t) & u_1(t) & \cdots &u_{2\M}(t)\\
        u_0(\pos_1) & u_1(\pos_1) & \cdots &u_{2\M}(\pos_1)\\
        u_0'(\pos_1) & u_1'(\pos_1) & \cdots &u_{2\M}'(\pos_1)\\
        \vdots & \vdots & & \vdots\\
        u_0(\pos_{\M}) & u_1(\pos_{\M}) & \cdots &u_{2\M}(\pos_{\M})\\
        u_0'(\pos_{\M}) & u_1'(\pos_{\M}) & \cdots &u_{2\M}'(\pos_{\M})
    \end{mydet},\label{eq:detV}\\
      \forall t\in \dompos\setminus\{\pos_0\},\quad      \detW(t)&\eqdef \frac{(2\M)!}{(t-\pos_0)^{2\M}} \begin{mydet}
        u_0(t) & u_1(t) & \cdots &u_{2\M}(t)\\
        u_0(\pos_1) & u_1(\pos_1) & \cdots &u_{2\M}(\pos_1)\\
        u_0'(\pos_1) & u_1'(\pos_1) & \cdots &u_{2\M}'(\pos_1)\\
        \vdots & \vdots & & \vdots\\
      u_0^{(2\M-1)}(\pos_{\M}) & u_1^{(2\M-1)}(\pos_{\M}) & \cdots &u_{2\M}^{(2\M-1)}(\pos_{\M})
  \end{mydet}.\label{eq:detW}
    \end{align}

  \begin{lemma}\label{lem:vandermonde}
If $\{u_0,\ldots, u_{2\M}\}\subset\Cder{2}(\dompos)$, then $\detV$ has a continuous extension to $\dompos$, with 
\begin{align}\label{eq:Dxi}
      \detV(\pos_i)\eqdef  \begin{vmatrix}
        u_0(\pos_1) & u_1(\pos_1) & \cdots &u_{2\M}(\pos_1)\\
        u_0'(\pos_1) & u_1'(\pos_1) & \cdots &u_{2\M}'(\pos_1)\\
        \vdots & \vdots & & \vdots\\
        u_0(\pos_i) & u_1(\pos_i) & \cdots &u_{2\M}(\pos_i)\\
        u_0'(\pos_i) & u_1'(\pos_i) & \cdots &u_{2\M}'(\pos_i)\\
        u_0''(\pos_i) & u_1''(\pos_i) & \cdots &u_{2\M}''(\pos_i)\\
        \vdots & \vdots & & \vdots\\
        u_0(\pos_{\M}) & u_1(\pos_{\M}) & \cdots &u_{2\M}(\pos_{\M})\\
        u_0'(\pos_{\M}) & u_1'(\pos_{\M}) & \cdots &u_{2\M}'(\pos_{\M})
      \end{vmatrix}.
    \end{align}
Similarly, if $\{u_0,\ldots, u_{2\M}\}\subset\Cder{2\M}(\dompos)$,
then $E$ has a continuous extension to $\dompos$, with 
    \begin{align}
      \detW(\pos_0)\eqdef \begin{vmatrix}
        u_0(\pos_1) & u_1(\pos_1) & \cdots &u_{2\M}(\pos_1)\\
        u_0'(\pos_1) & u_1'(\pos_1) & \cdots &u_{2\M}'(\pos_1)\\
        \vdots & \vdots & & \vdots\\
      u_0^{(2\M)}(\pos_{\M}) & u_1^{(2\M)}(\pos_{\M}) & \cdots &u_{2\M}^{(2\M)}(\pos_{\M})
      \end{vmatrix}.
    \end{align}
\end{lemma}

\begin{proof}
  This is a simple consequence of the Taylor expansions:
  \begin{align*}
    u_{j}(t)&=u_j(\pos_i)+u'(\pos_i)(t-\pos_i) +u''(\pos_i)\frac{(t-\pos_i)^2}{2} +o\left((t-\pos_i)^2\right),\\
    u_{j}(t)&=u_j(\pos_0)+u'(\pos_0)(t-\pos_0)+\cdots+u^{(2\M)}(\pos_0)\frac{(t-\pos_0)^{2\M}}{(2\M)!} +o\left((t-\pos_i)^{2\M}\right),
  \end{align*}
  as $t\to \pos_i$ (resp. $\pos_0$).
\end{proof}

\subsection{Main theorem}
As explained in Section~\ref{sec:blasso}, the low noise stability of the support of solutions is governed by the \textit{vanishing derivatives precertificate}
\begin{equation*}
  \etaV=\Phi^*\pV \qwhereq \pV\eqdef \argmin\enscond{\normH{p}}{\Phi^*p(x_i)=1,\quad (\Phi^*p)'(\pos_i)=0},
\end{equation*}
provided $\Gamma_{\bpos}$ has full rank. From~\eqref{eq:etaVform}, we known that $\etaV$ can be written as 
\begin{align}\label{eq:etavexpr}
  \eta_V(\pos)=\sum_{i=1}^\M \left(\alpha_i\Co(\pos,\pos_i)+ \beta_i\partial_2\Co(\pos,\pos_i)\right),
\end{align}
where $\Co(\pos,y)\eqdef \dotpH{\varphi(\pos)}{\varphi(y)}$ is the autocorrelation function of $\varphi$.

In the following, we let 
\begin{align}
 \left(\uplain_0,\uplain_1,\uplain_2,\ldots ,\uplain_{2\M-1},\uplain_{2\M}\right)&\eqdef\left(1,\Co(\cdot,x_1),\partial_2\Co(\cdot,\pos_1),\ldots, \Co(\cdot,\pos_\M), \partial_2\Co(\cdot,\pos_M)\right),\label{eq:uplain}
\end{align}
and we define $\detV$ (resp. $\detW$) as in~\eqref{eq:detV} (resp.~\eqref{eq:detW}).

\begin{theorem}\label{thm:tsystem}
  Assume that $\varphi\in \kernel{2}$ and let $\mo=\sum_{i=0}^{\M}\amp_i\delta_{\pos_i}$, with $\amp_i>0$ for all $i\in\{1,\ldots, \M\}$. Assume moreover that $\Gamma_{\bpos}$ has full rank. 
  
  Then the Non-Degenerate Source condition (NDSC) holds for $\mo$ if and only if
  \begin{align}
    \label{eq:cdtdetV}
    \forall t\in \dompos, \quad \detV(t)>0.
  \end{align}

Similarly, let $\varphi\in \kernel{2\M}$ and assume that $\Fdn$ has full rank. 

Then, the $(2\M-1)$ Non Degenerate Source Condition holds at $\pos_0$ if and only if 
  \begin{align}
    \label{eq:cdtdetW}
    \forall t\in \dompos, \quad \detW(t)>0.
  \end{align}
\end{theorem}

The above conditions are clearly reminiscent of the \textsc{Determinantal} condition in \cite{schiebinger2015superresolutionjournal}. Three differences should be noted. First, it is an infinitesimal condition: we do not consider neighborhoods of the $\pos_i$'s, but instead we consider derivatives. That both makes the proof simpler and provides us with the non-vanishing of the second derivatives of $\etaV$. Second of all, the sign matters: we impose the positivity, which implies that the constructed function $\etaV$ is below $1$ (whereas the condition in~\cite{schiebinger2015superresolutionjournal} only states that either $\etaV$ or $1-\etaV$ is below $1$). Last, we do not introduce any normalization, and we choose $\uplain_0=1$.

\begin{proof}
  First, we recall that if $\Gamma_{\bpos}$ has full rank, $\etaV$ is the unique function of the form~\eqref{eq:etavexpr} such that $\etaV(\pos_i)=1$, $\etaV'(\pos_i)=0$.
  
Note that by Lemma~\ref{lem:blassocramer} and Lemma~\ref{lem:vandermonde}, we know that $\etaV''(\pos_i)<0$ if and only if  
$\detV(\pos_i)>0$. It remains to show that for $t\in\dompos\setminus\{\pos_i\}_{i=1}^\M$,  $\detV(t)\neq 0$ if and only if $\etaV(t)\neq 1$.

Suppose that $\detV(t)\neq 0$ and assume by contradiction that $\etaV(t)= 1$. Hence the vector $(-1,\alpha_1,\beta_1,\ldots,\alpha_\M,\beta_\M)$ is a non-trivial kernel element of the matrix in~\eqref{eq:detV}, which contradicts $\detV(t)\neq 0$.
  Conversely, if $\etaV(t)\neq 1$, assume by contradiction that $\detV(t)=0$.
There exists a non trivial vector $(\tilde{\alpha}_0,\tilde{\alpha}_1,\tilde{\beta_1},\cdots,\tilde{\alpha}_\M,\tilde{\beta_\M})$ in the kernel of the matrix in~\eqref{eq:detV}. If $\tilde{\alpha}_0=0$, this yields a non trivial vector in the kernel of $\Gamma_{\bpos}^*\Gamma_{\bpos}$, contradicting the fact that $\Gamma_{\bpos}$ has full rank.
  Hence, we may assume that $\tilde{\alpha}_0=-1$, hence the function 
  \begin{align*}
    \pos\mapsto \sum_{i=1}^\M \left(\tilde{\alpha}_i\Co(\pos,\pos_i)+ \tilde{\beta}_i\partial_2\Co(\pos,\pos_i)\right)
  \end{align*}
is equal to $\etaV$, and $(-1,\tilde{\alpha}_1,\tilde{\beta_1},\cdots,\tilde{\alpha}_\M,\tilde{\beta}_\M)$ being in the kernel of~\eqref{eq:detV} is equivalent to $\etaV(t)=1$, a contradiction.

To conclude, if $\detV(t)>0$ for all $t\in\dompos$, then by continuity $1-\etaV$ cannot vanish in any interval $(\inf\dompos,\pos_1)$, $(\pos_i,\pos_{i+1})$ or $(\pos_{M},\sup\dompos)$, and $\etaV''(\pos_i)<0$ for $1\leq i\leq \M$. As a result $1-\etaV(t)<1$ for all $t\in \dompos\setminus\{\pos_i\}_{i=1}^\M$ and the (NDSC) holds. Conversely, if  $1-\etaV(t)<1$ for all $t\in \dompos\setminus\{\pos_i\}_{i=1}^\M$ and $\etaV''(\pos_i)<0$, we obtain that $\detV(t)>0$ for  $1\leq i\leq \M$.

  Thus, the equivalence is proved in the case of $\etaV$ and $\detV$. The proof for $\etaW$ and $\etaW$ is the same, \textit{mutatis mutandis}, by observing that if $\Fdn$ has full rank, $\etaW$ is the unique function of the form
    \begin{align*}
      \pos\mapsto \sum_{k=1}^{2\M-1} {\alpha}_k(\partial_2)^k \Co(\pos,\pos_0)
  \end{align*}
  which satisfies $\etaW(\pos_0)=1$, $\etaW^{(k)}=0$ for $k\in \{1,\ldots, 2\M-1\}$.
\end{proof}

\begin{rem}
  Note that the condition $\detV(t)>0$ and $\detW(t)>0$ are satisfied for all $t\in \dompos$ in the case where the family
  \begin{align}
    \begin{pmatrix}
      1 & \Co(\cdot,\pos_1) & \partial_2\Co(\cdot,\pos_1) &\ldots & \Co(\cdot,\pos_\M) & \partial_2\Co(\cdot,\pos_\M) 
    \end{pmatrix}
  \end{align}
  is an \textit{extended T-system} in the sense of~\cite{karlin1966tchebycheff}.
\end{rem}

\subsection{The Cauchy-Binet formula}
\label{sec:cauchy}
As noticed in~\cite{schiebinger2015superresolutionjournal}, the Cauchy-Binet formula (also called \textit{basic composition formula}~\cite[Ch.~1, example 8]{karlin1966tchebycheff})
is a powerful tool to prove recovery results for the \blassosc. In this section, we describe how it yields sufficient conditions for the criteria exhibited above.

As explained in Section~\ref{sec:blassoreconstruction}, we consider reconstruction frameworks which are determined by $\Hh=\LDP$, $\varphi(\pos)=\s\mapsto \psi(\pos,\s)$, so that 
\begin{align*}
  \Co(x,x')\eqdef\dotpH{\varphi(\pos)}{\varphi(\pos')}=\int_{\domobs}\psi(\pos,\s)\psi(\pos',\s)\d\Ps(\s). 
\end{align*}

We denote by $\Ps^{\otimes \M}$ the product measure of $\Ps$ on $\domobs^{\M}$, and we define 
\begin{align}
  \TS&\eqdef\enscond{(\s_1,\ldots, \s_{\M})\in \domobs^{\M}}{\s_1<\ldots <\s_{\M}}.
\end{align}

Since we extensively use the Cauchy-Binet formula and variants of it in the next paragraphs, let us recall its principle. Given $\theta,\xi\in\LDP$, one computes the following determinant by using linearity along each column
\begin{align}
  & \begin{vmatrix}
    \int_{\domobs}\theta(t_1,\s)\xi(u_1,\s)\d\Ps(\s) &\cdots &\int_{\domobs}\theta(t_1,\s)\xi(u_\M,\s)\d\Ps(\s)\\
    \vdots & & \vdots\\
    \int_{\domobs}\theta(t_\M,\s)\xi(u_1,\s)\d\Ps(\s) &\cdots &\int_{\domobs}\theta(t_\M,\s)\xi(u_\M,\s)\d\Ps(\s)
\end{vmatrix}\label{eq:cauchybinet1}\\
  &=\int_{\domobs}\ldots\int_{\domobs} \begin{vmatrix}
  \theta(t_1,\s_1)\xi(u_1,\s_1) &\cdots &\theta(t_1,\s_\M)\xi(u_\M,\s_\M)\\
    \vdots & & \vdots\\
    \theta(t_\M,\s_1)\xi(u_1,\s_1) &\cdots &\theta(t_\M,\s_\M)\xi(u_\M,\s_\M)
\end{vmatrix}\d\Ps(\s_1)\ldots \d\Ps(\s_\M)\nonumber\\
&=\int_{\domobs^\M}\xi(u_1,\s_1)\ldots \xi(u_\M,\s_\M)\begin{vmatrix}
  \theta(t_1,\s_1) &\cdots &\theta(t_1,\s_\M)\\
    \vdots & & \vdots\\
    \theta(t_\M,\s_1) &\cdots &\theta(t_\M,\s_\M)
\end{vmatrix}\d\Ps^{\otimes\M}(\s_1,\ldots,\s_\M).\nonumber
\end{align}
Now, the determinant vanishes if at least two $\s_i$ values are equal. Removing such $\M$-plets, the remaining $\M$-plets may be written as $(\s_{\sigma(1)},\ldots,\s_{\sigma(\M)})$ for some vector $(\s_1,\ldots,\s_\M)\in\TS$ and some permutation $\sigma$ of $\{1, \ldots,\M \}$. Splitting the domain $\domobs^\M$ along all such permutations and reordering the columns of the determinant in $\theta$, we obtain that~\eqref{eq:cauchybinet1} is equal to 
\begin{align}
 \int_{\TS}\begin{vmatrix}
   \xi(u_1,\s_1) & \cdots & \xi(u_1,\s_\M)\\
   \vdots & & \vdots\\
   \xi(u_\M,\s_1) & \cdots & \xi(u_\M,\s_\M)
 \end{vmatrix}\begin{vmatrix}
  \theta(t_1,\s_1) &\cdots &\theta(t_1,\s_\M)\\
    \vdots & & \vdots\\
    \theta(t_\M,\s_1) &\cdots &\theta(t_\M,\s_\M)
\end{vmatrix}\d\Ps^{\otimes\M}(\s_1,\ldots,\s_\M). \label{eq:cauchybinet2}
\end{align}
We note that the integration domain satisfies $\Ps^{\otimes \M}(\TS)>0$ if and only if $\supp(\Ps)$ has at least $\M$ points.

Now, we introduce several determinants which are helpful to prove that the conditions of Theorem~\ref{thm:tsystem} hold.
\begin{align}
  \label{eq:defAV}
  \detAV(\s_1,\ldots, \s_{2\M}) &\eqdef \begin{mydet}
    \psi(\pos_1, \s_1) & \partial_1\psi(\pos_1,\s_1) & \cdots &  \psi(\pos_\M, \s_1) & \partial_1\psi(\pos_\M,\s_1)\\
    \vdots & \vdots & & \vdots & \vdots\\
    \psi(\pos_1, \s_{2\M}) & \partial_1\psi(\pos_1,\s_{2\M}) & \cdots &  \psi(\pos_\M, \s_{2\M}) & \partial_1\psi(\pos_\M,\s_{2\M}) 
  \end{mydet},\\
  \label{eq:defBV}
  \mbox{and } \detBV(t,\s_1,\ldots, \s_{2\M}) &\eqdef \frac{1}{\prod_{i=1}^{\M}(t-\pos_i)^2}\begin{mydet}
  1 & \psi(t,\s_1) & \cdots & \psi(t,\s_{2\M})\\  
  1 & \psi({\pos_1},\s_1) & \cdots & \psi(\pos_{1},\s_{2\M})\\  
  0 & \partial_1\psi({\pos_1},\s_1) & \cdots & \partial_1\psi(\pos_{1},\s_{2\M})\\  
  \vdots & \vdots & & \vdots\\
  1 & \psi({\pos_{\M}},\s_1) & \cdots & \psi(\pos_{\M},\s_{2\M})\\  
  0 & \partial_1\psi(\pos_{\M},\s_1) & \cdots & \partial_1\psi(\pos_{\M},\s_{2\M})\\  
  \end{mydet}.
\end{align}
Similarly to~\eqref{eq:Dxi}, $\detBV$ may be extended by continuity at each $t=\pos_i$ with a determinant whose first row is $\begin{pmatrix}
  0 &  \left((\partial_1)^2\psi(\pos_i,\s_j)\right)_{1\leq j\leq 2\M}
\end{pmatrix}$.

By an elementary variant of the Cauchy-Binet formula, one may prove that 
\begin{align*}
  \det(\Ga_{\pos}^*\Ga_{\pos})= \int_{\TSD} \left(\detAV(\s_1,\ldots,\s_{2\M})\right)^2\d\Ps^{\otimes 2\M}(\s_1,\ldots,\s_{2\M}),\\
  \detV(t)=\int_{\TSD}\detAV(\s_1,\ldots,\s_{2\M})\detBV(t,\s_1,\ldots,\s_{2\M})\d\Ps^{\otimes 2\M}(\s_1,\ldots,\s_{2\M}).
\end{align*}

We obtain 
\begin{proposition}\label{prop:basiccompo}
  If $\card(\supp(\Ps))\geq 2\M$, and 
  \begin{align}
    \mbox{$\Ps^{\otimes 2\M}$-a.e. $(\s_1,\ldots,\s_{2\M})\in \TSD$},\quad \detAV(\s_1,\ldots, \s_{2\M})   >0,
  \end{align} then $\Gamma_\bpos$ has full rank.

For all $t\in \dompos$, if moreover
  \begin{align}
   \mbox{$\Ps^{\otimes 2\M}$-a.e. $(\s_1,\ldots,\s_{2\M})\in \TSD$},\quad \detBV(t,\s_1,\ldots, \s_{2\M})>0,
  \end{align}
then  $\detV(t)>0$.
\end{proposition}

Similarly, we may define  
\begin{align}
   \label{eq:defAW}
  \detAW(\s_1,\ldots, \s_{2\M}) &\eqdef \begin{mydet}
    \psi(\pos_0, \s_1) & \partial_1\psi(\pos_0,\s_1) & \cdots &  (\partial_1)^{2\M-1}\psi(\pos_0, \s_1)\\
    \vdots & \vdots & & \vdots \\
    \psi(\pos_0, \s_{2\M}) & \partial_1\psi(\pos_0,\s_{2\M}) & \cdots &  (\partial_1)^{2\M-1}\psi(\pos_0, \s_{2\M})
  \end{mydet},\\
    \label{eq:defBW}
    \mbox{and } \detBW(t,\s_1,\ldots, \s_{2\M}) &\eqdef \frac{1}{(t-\pos_0)^{2\M}}\begin{mydet}
    1 & \psi(t,\s_1) & \cdots & \psi(t,\s_{2\M})\\  
    1 & \psi({\pos_0},\s_1) & \cdots & \psi(\pos_{0},\s_{2\M})\\  
    0 & \partial_1\psi({\pos_0},\s_1) & \cdots & \partial_1\psi(\pos_{0},\s_{2\M})\\  
  \vdots & \vdots & & \vdots\\
    0 & (\partial_1)^{2\M-1}\psi({\pos_{0}},\s_1) & \cdots & (\partial_1)^{2\M-1}\psi(\pos_{0},\s_{2\M})
  \end{mydet}.
\end{align}
As above, one may prove that 
\begin{align*}
  \det(\Fdn^*\Fdn)= \int_{\TSD} \left(\detAW(\s_1,\ldots,\s_{2\M})\right)^2\d\Ps^{\otimes 2\M}(\s_1,\ldots,\s_{2\M}),\\
  \detW(t)=\int_{\TSD}\detAW(\s_1,\ldots,\s_{2\M})\detBW(t,\s_1,\ldots,\s_{2\M})\d\Ps^{\otimes 2\M}(\s_1,\ldots,\s_{2\M}).
\end{align*}
thus 
\begin{proposition}\label{prop:basiccompoetaW}
  If $\card(\supp(\Ps))\geq 2\M$, and 
  \begin{align}
    \mbox{$\Ps^{\otimes 2\M}$-a.e. $(\s_1,\ldots,\s_{2\M})\in \TSD$},\quad \detAW(\s_1,\ldots, \s_{2\M})   >0,
  \end{align} then $\Gamma_\bpos$ has full rank.

For all $t\in \dompos$, if moreover
  \begin{align}
   \mbox{$\Ps^{\otimes 2\M}$-a.e. $(\s_1,\ldots,\s_{2\M})\in \TSD$},\quad \detBW(t,\s_1,\ldots, \s_{2\M})>0,
  \end{align}
then  $\detW(t)>0$.
\end{proposition}

\subsection{Extension to the unconstrained Lasso}
While the main focus of this paper is the Lasso with positivity constraint, let us note that the characterization of the NDSC provided in Theorem~\ref{thm:tsystem} can be adapted to the unconstrained Lasso. If $m_0=\sum_{i=1}^\M a_i\delta_{\pos_i}$, let $I^+=\enscond{i\in \{1,\ldots,\M\}}{a_i>0}$ and $I^-=\enscond{i\in \{1,\ldots,\M\}}{a_i<0}$. The Non-Degenerate Source Condition on $\etaV$ is that 
\begin{align}
  \forall i\in I^+,\quad  \etaV(\pos_i) &=1, \quad \etaV''(\pos_i)<0,\\
  \forall i\in I^-,\quad  \etaV(\pos_i) &=-1, \quad \etaV''(\pos_i)>0,\\
  \forall t\in \dompos\setminus\{\pos_i\}_{i=1}^\M, \quad \abs{\etaV(t)}&<1.
\end{align}
Using the same arguments as above, it is possible to prove that the Non-Degenerate Source Condition is equivalent to
\begin{align}
  \forall t\in \dompos,\quad  \detV^+(t)&>0 \qandq \detV^-(t)<0.
\end{align}
where 
\begin{align*}
  \detV^+(t)&\eqdef \frac{2}{\prod_{i\in I^+}(t-\pos_i)^2}\\
            &\times\begin{mydet}
    1 & \Co(t,\pos_1) & \partial_2\Co(t,\pos_1) & \cdots &\Co(t,\pos_{\M}) & \partial_2\Co(t,\pos_{\M})\\
    \sign(a_1)  & \Co(\pos_1,\pos_1) & \partial_2\Co(\pos_1,\pos_1) & \cdots &\Co(\pos_1,\pos_{\M}) & \partial_2\Co(\pos_1,\pos_{\M})\\
    0  & \partial_1\Co(\pos_1,\pos_1) & \partial_1\partial_2\Co(\pos_1,\pos_1) & \cdots &\partial_1\Co(\pos_1,\pos_{\M}) & \partial_1\partial_2\Co(\pos_1,\pos_{\M})\\
  \vdots & \vdots &\vdots& & \vdots &\vdots\\
    \sign(a_\M)  & \Co(\pos_\M,\pos_1) & \partial_2\Co(\pos_\M,\pos_1) & \cdots &\Co(\pos_\M,\pos_{\M}) & \partial_2\Co(\pos_\M,\pos_{\M})\\
    0  & \partial_1\Co(\pos_{\M},\pos_1) & \partial_1\partial_2\Co(\pos_{\M},\pos_1) & \cdots &\partial_1\Co(\pos_{\M},\pos_{\M}) & \partial_1\partial_2\Co(\pos_{\M},\pos_{\M})
  \end{mydet},\\
  \detV^-(t)&\eqdef \frac{2}{\prod_{i\in I^-}(t-\pos_i)^2}\\
            &\times \begin{mydet}
    -1 & \Co(t,\pos_1) & \partial_2\Co(t,\pos_1) & \cdots &\Co(t,\pos_{\M}) & \partial_2\Co(t,\pos_{\M})\\
    \sign(a_1)  & \Co(\pos_1,\pos_1) & \partial_2\Co(\pos_1,\pos_1) & \cdots &\Co(\pos_1,\pos_{\M}) & \partial_2\Co(\pos_1,\pos_{\M})\\
    0  & \partial_1\Co(\pos_1,\pos_1) & \partial_1\partial_2\Co(\pos_1,\pos_1) & \cdots &\partial_1\Co(\pos_1,\pos_{\M}) & \partial_1\partial_2\Co(\pos_1,\pos_{\M})\\
  \vdots & \vdots &\vdots& & \vdots &\vdots\\
    \sign(a_\M)  & \Co(\pos_\M,\pos_1) & \partial_2\Co(\pos_\M,\pos_1) & \cdots &\Co(\pos_\M,\pos_{\M}) & \partial_2\Co(\pos_\M,\pos_{\M})\\
    0  & \partial_1\Co(\pos_{\M},\pos_1) & \partial_1\partial_2\Co(\pos_{\M},\pos_1) & \cdots &\partial_1\Co(\pos_{\M},\pos_{\M}) & \partial_1\partial_2\Co(\pos_{\M},\pos_{\M})
  \end{mydet}.
\end{align*}


\section{The case of Laplace and Gaussian measurements}\label{sec:laplacegauss}
Now, we discuss the applicability of Theorem~\ref{thm:tsystem} to Laplace and Gaussian measurements.

\subsection{Laplace measurements $\psi(\pos,\s)=e^{-\pos\s}$}\label{sec:laplace}
  It is standard that the exponential kernel $K(x,y)=e^{xy}$ is \textit{extended totally positive} (see~\cite[Ch. 3]{karlin1968total}), that is for all $N\in \NN^*$, all real numbers $x_1< \ldots< x_N$, $x'_1< \ldots < x'_N$, there holds $\det((K(x_i,x'_j))_{1\leq i,j\leq N})>0$. If $x_i=x_{i+1}=\ldots=x_{i+r-1}$, then one may replace the corresponding rows with the successive derivatives
\begin{align}\label{eq:replaceET}
    \begin{pmatrix}
     K(x_i,x'_1) & \cdots &  K(x_i,x'_N)\\
    \partial_1 K(x_i,x'_1) & \cdots & \partial_1 K(x_i,x'_N)\\
    \vdots & & \vdots \\
    (\partial_1)^{r-1} K(x_i,x'_1) & \cdots & (\partial_1)^{r-1} K(x_i,x'_N)
    \end{pmatrix}
  \end{align}
 and still get a strictly positive quantity. The same holds w.r.t $x'$.
  
  Changing $y$ into $-\s$, setting $\psi(t,\s)\eqdef K(t,-\s)$ and reordering the rows, we see that the sign of $\det((\exp(-t_i\s_j))_{1\leq i,j\leq N})$ is $(-1)^{\frac{N(N-1)}{2}}$ if $t_1\leq \ldots \leq t_N$ and $s_1\leq \ldots \leq s_N$ (and the usual adaptation if some $t_i$ or $\s_j$ values are equal). 

Our first result in this direction is the following.

\begin{corollary}\label{coro:laplace}
  Let $\dompos=[c,+\infty)$ with $c\geq 0$, $\domobs\subset (0,+\infty)$ and $\Ps$ be a positive measure such that $\int_{\domobs} (1+|\s|)^{4\M}e^{-2c\s}\d\Ps(s)<+\infty$. 
    
    If $\psi(\pos,\s)=e^{-\pos\s}$ and $\card(\supp(\Ps))\geq 2\M$, the following holds.
    \begin{itemize}
      \item If $m_0=\sum_{i=1}^{\M}a_i\delta_{\pos_i}$, with $\pos_i\in\interop(\dompos)$ and $\amp_i>0$ for all $i$, then $m_0$ satisfies the Non-Degenerate Source Condition.
    \item If $\pos_0\in \interop(\dompos)$, then the point $\pos_0$ is $(2\M-1)$ Non Degenerate.
  \end{itemize}
\end{corollary}
The proof is given in Appendix~\ref{sec-prooflap}.

\begin{rem}
  One might wonder why we impose $\domobs\subset (0,+\infty)$ rather than $\domobs\subset [0,+\infty)$. It is due to the artificial point $\s_0$ which is introduced. It is possible to choose $\domobs= [0,+\infty)$ and $0\in \supp(\Ps)$, but then $2\M$ measurements are not sufficient anymore, since for $\s_1=0=\s_0$, the function $t\mapsto \detBV(t,\s_1,\ldots,\s_{2\M})$ vanishes identically. As a result, $\etaV$ (resp. $\etaW$) is identically equal to $1$ and the (NDSC) does not hold.
      
      In that case, the conclusion of Proposition~\ref{coro:laplace} still holds with the stronger  condition that $\card(\supp(\Ps))\geq 2\M+1$; everything happens as if the measurement at $\s=0$ were ignored.
\end{rem}


\subsection{Gaussian kernel $\psi(\pos,\s)=e^{-(\pos-\s)^2}$}\label{sec:gauss}
Now, we deal with the more involved case of the Gaussian kernel. As noted in~\cite{schiebinger2015superresolutionjournal}, a family of the form  $\{1,e^{-(\cdot-\pos_1)^2},\ldots,e^{-(\cdot-\pos_\M)^2 }\}$ does not form a $T$-system, hence the sufficient condition edicted in~\cite{deCastro-exact2012} for identifiability does not hold. For that reason, the authors of~\cite{schiebinger2015superresolutionjournal} were led to introduce a weighted total variation (equivalent to some renormalization of the impulse responses) to provide identifiability. 

In our context, $\detBV(t,\s_1,\ldots,\s_{2M})>0$ does not always hold, hence it is not possible to assert that the conditions of Theorem~\ref{thm:tsystem} hold regardless of the sampling $(\domobs,\Ps)$. In fact, the experiments shown in Figures~\ref{fig:gauss} and~\ref{fig:confinedgauss} indicate that for some samplings measures $\Ps$, the problem is not support stable (is not clear to us whether or not $m_0$ is identifiable in such cases).

However, as described below, it is possible to assert support stability in two cases which are worthy of interest. The first one is the regular convolution with a Gaussian kernel, without sampling, or with some sampling which approximates the Lebesgue measure (in a suitable weak sense). The second one is when the sampling is confined into a small interval.

\subsubsection{Full sampling, or almost}
When fully observing the convolution of $m_0$ with some Gaussian kernel, \ie when
\begin{align}
  \phi(\x): \RR\rightarrow \LD(\RR),\quad  \s\mapsto \int_{\RR} e^{-(\pos-\s)^2}\d m_0(\pos), 
\end{align}
the operator $\Phi$ is injective, hence it is clear that $m_0$ is the unique solution to~\eqref{eq:bpursuit}. One might wonder, however, whether this recovery is support stable.

\begin{proposition}[Fully sampled Gaussian kernel]\label{prop:fullgaussian}
  Let $\Ps$ be the Lebesgue measure on $\RR$, $\dompos=\RR$ and $\psi(\pos,\s)\eqdef e^{-(\pos-\s)^2}$. Then
    \begin{itemize}
    \item Any measure $m_0=\sum_{i=1}^{\M}a_i\delta_{\pos_i}$, with $\pos_i\in \RR$ and $\amp_i>0$ for all $i$, satisfies the Non-Degenerate Source Condition.
    \item For any point $\pos_0\in\RR$, then the point $\pos_0$ is $(2\M-1)$ Non Degenerate.
  \end{itemize}
\end{proposition}

A proof of Proposition~\ref{prop:fullgaussian} is given in Appendix~\ref{sec:prooffullgauss}.
Let us note that this proposition is a consequence of the following technical lemma proved in~\cite{schiebinger2015superresolutionjournal}.

\begin{lemma}[\protect{\cite[Lemma 2.7]{schiebinger2015superresolutionjournal}}]\label{lem:polygauss}
  Let $d_1,\ldots d_r\in \NN$ and $\s_1<\ldots <\s_r$ be real numbers. The function 
  \begin{align}
    \t\mapsto \sum_{i=1}^r (a_{i,0}+a_{i,1}\t +\cdots +a_{i,(2d_i-1)}\t^{(2d_i-1)})e^{\s_it} + e^{t^2}
  \end{align}
  has at most $2(d_1+\cdots +d_r)$ zeros, counting multiplicities.
\end{lemma}

As a consequence of Proposition~\ref{prop:fullgaussian}, we also obtain non-degeneracy results for measures with finite total mass which are sufficiently close to the uniform Lebesgue measure. More precisely, let us assume that
\begin{align}\label{eq:cvgauss}
  \lim_{n\to+\infty}\max_{0\leq k,\ell\leq 2\M} \sup_{x\in \dompos} \left|\int_{\RR} x^{k}\s^{\ell}e^{-(\pos-\s)^2-(\pos_i-\s)^2}\d\Pns(\s)- \int_{\RR}  x^{k}\s^{\ell}e^{-(\pos-\s)^2-(\pos_i-\s)^2}\d\s\right|=0.
\end{align}

\begin{proposition}\label{prop:samplebesque}
  Let $\dompos\subseteq\RR$, $(\Pns)_{n\in\NN}$ be a sequence of positive measures with finite total mass such that~\eqref{eq:cvgauss} holds.
  Let  $m_0=\sum_{i=1}^{\M}a_i\delta_{\pos_i}$, with $\pos_i\in \interop(\dompos)$ and $\amp_i>0$ for all $i$ (resp. let $\pos_0\in \interop(\dompos)$).
  
Then, provided $n$ is large enough, 
 \begin{enumerate}
   \item[i)] 
$m_0$ satisfies the  Non-Degenerate Source Condition,
\item[ii)]
The point $\pos_0$ is $2\M-1$ Non-Degenerate,
 \end{enumerate}
for the reconstruction framework defined by $\Pns$ and $\psi(\pos,\s)= e^{-(\pos-\s)^2}$.
\end{proposition}

A proof of Proposition~\ref{prop:samplebesque} is given in Appendix~\ref{sec:proofsamplebesque}.

\subsubsection{Confined sampling} 
Now, we return to the framework of Section~\ref{sec:cauchy} with some sampling measure $\Ps$. We show that if the support of $\Ps$ is confined in some sufficiently small interval, stable recovery of the support is possible regardless of the spikes separation.

\begin{proposition}\label{prop:confined}
  Let $\dompos\subseteq \RR$ be an interval, $\sst\in \RR$ and $\Ps$ be a positive measure such that $\card(\supp(\Ps))\geq 2\M$. There exists $r>0$ such that if $\supp \Ps\subset [\sst-r,\sst+r]$, then for any measure  $m_0=\sum_{i=1}^{\M}a_i\delta_{\pos_i}$, with $\pos_i\in \interop(\dompos)$ and $\amp_i>0$ for all $i$ (resp. for any $\pos_0\in \interop(\dompos)$),
 \begin{enumerate}
   \item[i)] 
$m_0$ satisfies the  Non-Degenerate Source Condition,
\item[ii)]
The point $\pos_0$ is $2\M-1$ Non-Degenerate,
 \end{enumerate}
in the reconstruction framework defined by $\Ps$ and $\psi(\pos,\s)= e^{-(\pos-\s)^2}$.
\end{proposition}

A proof of Proposition~\ref{prop:confined} is given in Appendix~\ref{sec:proofconfined}.

\section{Renormalizing the atoms}
In this last section, we discuss the situation when renormalizing the atoms, which amounts to the setting of~\cite{schiebinger2015superresolutionjournal}.
\label{sec:normalization}
Instead of directly using $\varphi$  and the corresponding correlation $\Co(x,y)$ in~\eqref{eq:bpursuit} and~\eqref{blasso}, one may wish to renormalize the atoms and consider $\phib(x)\eqdef \frac{1}{\Nf(x)}\varphi(x)$ and the corresponding correlation $\Cob(x,y)=\frac{1}{\Nf(x)\Nf(y)}\Co(x,y)$, where $x\mapsto \Nf(x)$ is a smooth (strictly) positive function. For instance $\Nf(x)=\sqrt{\Co(x,x)}$ yields atoms $\phib(x)$ with the same norm in $\Hh$.

It is then important to understand when the Non-Degenerate Source Condition holds for $\Cob$.

For notational simplicity, we define
\begin{align}
 \left(\unorm_0,\unorm_1,\unorm_2,\ldots ,\unorm_{2\M-1},\unorm_{2\M}\right)&\eqdef\left(1,\Cob(\cdot,x_1),\partial_2\Cob(\cdot,\pos_1),\ldots, \Cob(\cdot,\pos_\M), \partial_2\Cob(\cdot,\pos_M)\right).\label{eq:unorm}
\end{align}
and we consider the associated $\detVbar$ defined through~\eqref{eq:detV}.

We notice that 
\begin{align}\label{eq:tsysleibniz}
  \partial_2\Cob(t,\pos_j)=\frac{1}{\Nf(t)\Nf(\pos_j)}\partial_2\Co(t,\pos_j) + \frac{\Co(t,\pos_j)}{\Nf(t)}\frac{\d}{\d \pos_j}\left(\frac{1}{\Nf(\pos_j)}\right)
\end{align}
therefore by subtracting a multiple of each even column to its successor, we see that for all $t\in \dompos\setminus\{\pos_i\}_{i=1}^{\M}$, 
\begin{align}
  &\begin{mydet}
    1 & {\Cob(t,\pos_1)} & {\partial_2\Cob(t,\pos_1)} & \ldots & {\Cob(t,\pos_\M)} & {\partial_2\Cob(t,\pos_\M)} \\
    1 & {\Cob(\pos_1,\pos_1)} & {\partial_2\Cob(\pos_1,\pos_1)}& \ldots & {\Cob(\pos_1,\pos_\M)} & {\partial_2\Cob(\pos_1,\pos_\M)} \\
    0 & {\partial_1\Cob(\pos_1,\pos_1)} & {\partial_1\partial_2\Cob(\pos_1,\pos_1)}& \ldots & {\partial_1\Cob(\pos_1,\pos_\M)} & {\partial_1\partial_2\Cob(\pos_1,\pos_\M)} \\
  \vdots & \vdots & \vdots & & \vdots & \vdots\\
  1 & {\Cob(\pos_{\M},\pos_1)} & {\partial_2\Cob(\pos_{\M},\pos_1)}& \ldots & {\Cob(\pos_{\M},\pos_\M)} & {\partial_2\Cob(\pos_{\M},\pos_\M)} \\
  0 & {\partial_1\Cob(\pos_{\M},\pos_1)} & {\partial_1\partial_2\Cob(\pos_{\M},\pos_1)}& \ldots & {\partial_1\Cob(\pos_{\M},\pos_\M)} & {\partial_1\partial_2\Cob(\pos_{\M},\pos_\M)} \\
\end{mydet}\nonumber\\  
  &=\frac{1}{\left(\prod_{j=1}^{\M}\Nf(\pos_j)\right)^2}
  \begin{mydet}
    1 & \frac{\Co(t,\pos_1)}{\Nf(t)} & \frac{\partial_2\Co(t,\pos_1)}{\Nf(t)} & \ldots & \frac{\Co(t,\pos_\M)}{\Nf(t)} & \frac{\partial_2\Co(t,\pos_\M)}{\Nf(t)} \\
    1 & \frac{\Co(\pos_1,\pos_1)}{\Nf(\pos_1)} & \frac{\partial_2\Co(\pos_1,\pos_1)}{\Nf(\pos_1)} & \ldots & \frac{\Co(\pos_1,\pos_\M)}{\Nf(\pos_1)} & \frac{\partial_2\Co(\pos_1,\pos_\M)}{\Nf(\pos_1)} \\
    0 & \frac{\partial_1\Co(\pos_1,\pos_1)}{\Nf(\pos_1)} & \frac{\partial_1\partial_2\Co(\pos_1,\pos_1)}{\Nf(\pos_1)} & \ldots & \frac{\partial_1\Co(\pos_1,\pos_\M)}{\Nf(\pos_1)} & \frac{\partial_1\partial_2\Co(\pos_1,\pos_\M)}{\Nf(\pos_1)} \\
  \vdots & \vdots & \vdots & & \vdots & \vdots\\
    1 & \frac{\Co(\pos_{\M},\pos_1)}{\Nf(\pos_\M)} & \frac{\partial_2\Co(\pos_{\M},\pos_1)}{\Nf(\pos_{\M})} & \ldots & \frac{\Co(\pos_1,\pos_\M)}{\Nf(\pos_\M)} & \frac{\partial_2\Co(\pos_\M,\pos_\M)}{\Nf(\pos_\M)} \\
    0 & \frac{\partial_1\Co(\pos_\M,\pos_1)}{\Nf(\pos_\M)} & \frac{\partial_1\partial_2\Co(\pos_\M,\pos_1)}{\Nf(\pos_\M)} & \ldots & \frac{\partial_1\Co(\pos_\M,\pos_\M)}{\Nf(\pos_\M)} & \frac{\partial_1\partial_2\Co(\pos_\M,\pos_\M)}{\Nf(\pos_\M)}
  \end{mydet}\\
  &=\frac{1}{\Nf(t)\left(\prod_{j=1}^{\M}\Nf(\pos_j)\right)^4}\\
  &\times\begin{mydet}
  \Nf(t) & \Co(t,\pos_1) & \partial_2\Co(t,\pos_1) & \ldots & \Co(t,\pos_\M) & \partial_2\Co(t,\pos_\M) \\
  \Nf(\pos_1) & \Co(\pos_1,\pos_1) & \partial_2\Co(\pos_1,\pos_1) & \ldots & \Co(\pos_1,\pos_\M) & \partial_2\Co(\pos_1,\pos_\M) \\
    0 & \partial_1\Co(\pos_1,\pos_1) & \partial_1\partial_2\Co(\pos_1,\pos_1) & \ldots & \partial_1\Co(\pos_1,\pos_\M) & \partial_1\partial_2\Co(\pos_1,\pos_\M) \\
  \vdots & \vdots & \vdots & & \vdots & \vdots\\
  \Nf(\pos_\M) & \Co(\pos_{\M},\pos_1) & \partial_2\Co(\pos_{\M},\pos_1) & \ldots & \Co(\pos_1,\pos_\M) & \partial_2\Co(\pos_\M,\pos_\M) \\
    0 & \partial_1\Co(\pos_\M,\pos_1) & \partial_1\partial_2\Co(\pos_\M,\pos_1) & \ldots & \partial_1\Co(\pos_\M,\pos_\M) & \partial_1\partial_2\Co(\pos_\M,\pos_\M)
  \end{mydet}\label{eq:tsyscob}
\end{align}

The above matrix, divided by $\prod_{i=1}^{\M}(t-\pos_i)^2$ is a infinitesimal version of the matrix $\Lambda(t,\pos_1,\pos_1, \ldots,\pos_\M, \pos_{\M})$ which appears in the \textsc{Determinantal} condition of~\cite{schiebinger2015superresolutionjournal}. We note that it is not specific to the $L^1$ norm, as $\Nf$ could be any (positive smooth) normalizing factor.

\begin{rem}\label{rem:tsysgammab}\label{rem:tsyspsib}
 If we define $\Gammab_{\bpos}\eqdef \begin{pmatrix}
   \phib(\pos_1) & \cdots &\phib(\pos_\M) & \phib'(\pos_1) &\cdots &\phib'(\pos_\M)
 \end{pmatrix}$, we note in view of~\eqref{eq:tsysleibniz} that $\det(\Gamma_{\bpos}^*\Gamma_{\bpos})$ and $\det(\Gammab_{\bpos}^*\Gammab_{\bpos})$ are equal up to a (strictly) positive multiplicative factor. Hence, when working with normalized atoms, it is sufficient to check that $\Gamma_{\bpos}$ has full rank (on the unnormalized family).
Similarly, if we define $\Fkb\eqdef \begin{pmatrix}
   \phib(\pos_0) & \phib'(\pos_0) &\cdots &\phib^{(k)}(\pos_0)
 \end{pmatrix}$, we note that $\det(\Fk^*\Fk)$ and $\det(\Fkb^*\Fkb)$ are equal up to a (strictly) positive multiplicative factor. Hence, when working with normalized atoms, it is sufficient to check that $\Fk$ has full rank (on the unnormalized family).
\end{rem}

\subsection{Cauchy-Binet formula}
In the rest of this section, to simplify things, we assume that $\Ps$ is a finite measure, \ie $\Ps(\domobs)<+\infty$. This ensures that the constant function $\s\mapsto 1$ is in $\Hh$ (and $\psi(t,\cdot)\in\LUP$ for all $t\in \dompos$). 

To ensure that~\eqref{eq:tsyscob} is positive, one may use the same decomposition as in~\cite{schiebinger2015superresolutionjournal} that it is equal up to a positive factor to
\begin{align}\label{eq:cauchybinetnorm}
  \int_{\TSDU}&\detAVbar(\s_0,\ldots,\s_{2\M})\detBVbar(t,\s_0,\ldots,\s_{2\M})\d\Ps^{\otimes 2\M+1}(\s_0,\ldots,\s_{2\M}),
\end{align}
where
\begin{align}
\detAVbar(\s_0,\ldots,\s_{2\M})&\eqdef \begin{mydet}
1 & \psi(\pos_1,\s_0) & \partial_1\psi(\pos_1,\s_0) & \cdots & \psi(\pos_\M,\s_0) & \partial_1\psi(\pos_\M,\s_0)\\
\vdots & \vdots & \vdots & & \vdots &\vdots\\
1 & \psi(\pos_1,\s_{2\M}) & \partial_1\psi(\pos_1,\s_{2\M}) & \cdots & \psi(\pos_\M,\s_{2\M}) & \partial_1\psi(\pos_\M,\s_{2\M})
\end{mydet}\\
\detBVbar(t,\s_0,\ldots,\s_{2\M})&\eqdef \frac{2}{\prod_{i=1}^\M(t-\pos_i)^2}\begin{mydet}
\psi(t,\s_0) &  \cdots & \psi(t,\s_{2\M})\\
\psi(\pos_1,\s_{0})  & \cdots & \psi(\pos_1,\s_{2\M}) \\
\partial_1\psi(\pos_1,\s_{2\M})& &\partial_1\psi(\pos_1,\s_{2\M})\\
\vdots & & \vdots\\
\psi(\pos_{\M},\s_{0})  & \cdots & \psi(\pos_{\M},\s_{2\M}) \\
\partial_1\psi(\pos_{\M},\s_{2\M})&\cdots&\partial_1\psi(\pos_{\M},\s_{2\M})
\end{mydet}
\end{align}

We obtain
\begin{proposition}\label{prop:normalcompo}
  Assume that $\Ps(\domobs)<+\infty$ and let $\Nf(t)\eqdef \dotp{1}{\phi(t)}=\int_{\domobs}\psi(t,\s)\d\Ps(\s)$. Assume moreover that $\Nf(t)>0$ for all $t\in\dompos$.
\begin{enumerate}
  \item   If $\card(\supp(\Ps))=2\M$, then~\eqref{eq:tsyscob} is identically zero. If moreover $\Gamma_{\bpos}$ has full rank, $\etaV$ is identically $1$.

  \item  If $\card(\supp(\Ps))\geq 2\M+1$, and $\detAVbar(\s_0,\ldots,\s_{2\M})>0$ for $\Ps^{\otimes 2\M+1}$-a.e. $(\s_0,\ldots,\s_{2\M})\in \TSDU$, 
 then $\Gammab_\bpos$ has full column rank.
 
 If, additionally,  for $t\in \dompos$, $\Ps^{\otimes 2\M+1}$-a.e. $(\s_0,\ldots,\s_{2\M})\in \TSDU$, $\detBVbar(t,\s_0,\ldots,\s_{2\M})>0$,
 then $\detVbar(t)>0$.
\end{enumerate}
 \end{proposition}
 We skip the proof as it follows directly from~\eqref{eq:cauchybinetnorm}. It should be noted that the $L^1$-normalization of the atoms ``spoils'' one measurement: one needs $2\M+1$ instead of $2\M$ to get identifiability or stability.

Again, there is a variant for the limit $(\pos_1,\ldots, \pos_\M)\to (\pos_0,\ldots, \pos_0)$. It relies on the decomposition 
\begin{align}\label{eq:cauchybinetnorm}
  \int_{\TSDU}&\detAWbar(\s_0,\ldots,\s_{2\M})\detBWbar(t,\s_0,\ldots,\s_{2\M})\d\Ps^{\otimes 2\M+1}(\s_0,\ldots,\s_{2\M}),
\end{align}
where
\begin{align}
\detAWbar(\s_0,\ldots,\s_{2\M})&\eqdef \begin{mydet}
  1 & \psi(\pos_0,\s_0) & \partial_1\psi(\pos_0,\s_0) & \cdots & (\partial_1)^{2\M-1}\psi(\pos_0,\s_0)\\
\vdots & \vdots & \vdots & & \vdots &\vdots\\
  1 & \psi(\pos_0,\s_{2\M}) & \partial_1\psi(\pos_0,\s_{2\M}) & \cdots &  (\partial_1)^{2\M-1}\psi(\pos_0,\s_{2\M})
\end{mydet}\\
\detBWbar(t,\s_0,\ldots,\s_{2\M})&\eqdef \frac{(2\M)!}{\prod_{i=1}^\M(t-\pos_i)^2}\begin{mydet}
\psi(t,\s_0) &  \cdots & \psi(t,\s_{2\M})\\
\psi(\pos_0,\s_{0})  & \cdots & \psi(\pos_0,\s_{2\M}) \\
\partial_1\psi(\pos_0,\s_{2\M})& \cdots &\partial_1\psi(\pos_0,\s_{2\M})\\
\vdots & & \vdots\\
(\partial_1)^{2\M-1}\psi(\pos_0,\s_{2\M})&\cdots&(\partial_1)^{2\M-1}\psi(\pos_0,\s_{2\M})
\end{mydet}
\end{align}
The conclusions of Proposition~\ref{prop:normalcompo} hold for $\etaW$ with the obvious adaptations.

\subsection{$L^1$-normalized Laplace measurements}
As a consequence we obtain the following result for $L^1$-normalized Laplace measurements, \ie  $\phib(\pos): \s\mapsto \frac{1}{\Nf(\pos)}\psi(\pos,\s)$ where $\psi(\pos,\s)=e^{-\pos\s}$.
\begin{corollary}
  Let $\dompos=[c,+\infty)$ where $c\geq 0$ and let $\Ps$ be a positive measure on $\domobs\subseteq[0,+\infty)$ such that $\int_{\domobs} (1+|\s|)^{4\M}e^{-2c\s}\d\Ps(s)<+\infty$ with $\card(\supp(\Ps))\geq 2\M+1$,
    \begin{itemize}
          \item If $m_0=\sum_{i=1}^{\M}a_i\delta_{\pos_i}$, with $\pos_i\in \interop(\dompos)$ and $\amp_i>0$ for all $i$, then $m_0$ satisfies the Non-Degenerate Source Condition,
          \item If $\pos_0\in \interop(\dompos)$, then the point $\pos_0$ is $(2\M-1)$ Non Degenerate,
  \end{itemize}
  for the reconstruction framework defined by $\Ps$ and $\overline{\psi}(\pos,\s)=\frac{1}{\Nf(\pos)}\psi(\pos,\s)$, where $\psi(\pos,\s)=e^{-\pos\s}$.
\end{corollary}

The proof is straightforward, introducing the point $\pos=0$ and using the extended total positivity of the exponential kernel.

\subsection{Arbitrary sampling for $L^1$ normalized Gaussian measurements}
The case of $L^1$ normalized Gaussian measurements, 
\ie  $\phib(\pos): \s\mapsto \frac{1}{\Nf(\pos)}\psi(\pos,\s)$, where 
$\psi(\pos,\s)=e^{-(\pos-\s)^2}$ for all $(\pos,\s)\in \RR^2$, and $\Nf(x)=\int_{\domobs}\psi(\pos,\s)\d\Ps(\s)$, is studied in~\cite{schiebinger2015superresolutionjournal}. The authors show that the alternative ($\etaV<1$ on $\dompos\setminus\{x_i\}_{i=1}^\M$) or ($\etaV>1$ on $\dompos\setminus\{x_i\}_{i=1}^\M$) holds, and they deduce the identifiability of non-negative sums of $\M$ Dirac masses. However it is not clear which of $\etaV> 1$ or $\etaV<1$ holds, and thus, whether the support is stable at low noise or not.

In fact, their proof contains all the necessary ingredients, as they show that both $\detAVbar$ and $\detBVbar$ are nonzero, and similarly for $\detAWbar$ and $\detBWbar$. By the same arguments as in the proof of Proposition~\ref{prop:fullgaussian}, one may prove that they are in fact positive and obtain
\begin{corollary}
  Let $\Ps$ be a finite positive measure on $\domobs\subseteq\RR$, with $\card(\supp(\Ps))\geq 2\M+1$.
  \begin{itemize}
    \item If $m_0=\sum_{i=1}^{\M}a_i\delta_{\pos_i}$, with $\pos_i\in\interop(\dompos)$ and $\amp_i>0$ for all $i$, then $m_0$ satisfies the Non-Degenerate Source Condition
    \item If $\pos_0\in \interop(\dompos)$, then the point $\pos_0$ is $(2\M-1)$ Non Degenerate
  \end{itemize}
  for the reconstruction framework defined by $\overline{\psi}(\pos,\s)=\frac{1}{\Nf(\pos)}\psi(\pos,\s)$ with $\psi(\pos,\s)=e^{-(\pos-\s)^2}$ for all $(\pos,\s)\in \RR^2$.
\end{corollary}


\section*{Conclusion}
In this work, we have provided a necessary and sufficient condition for the Non-Degenerate Source Condition which ensures support stability for the total variation sparse recovery of measures. We have proved that this condition holds for the partial Laplace transform, regardless of the repartition of the samples. In the case of the convolution with a Gaussian filter, we have shown that the proposed conditions hold in the case of the full observation of the convolution, as well as for samplings which approximate the Lebesgue measure. Additionally such conditions hold in the case of \textit{confined samplings}, when the convolution is observed in some small interval.

The main purpose of the present work is to lay the theoretical foundations of super-resolution with Laplace observations~\cite{Laplace} in view of the forthcoming paper~\cite{Laplace} which deals with Total Internal Reflection Microscopy (TIRF).

\section*{Acknowledgements}
The author would like to thank Gabriel Peyr\'e, Quentin Denoyelle and Emmanuel Soubies for stimulating discussions about this work, in particular for suggesting the development around the unsampled Gaussian convolution.

\appendix
\section{Proofs of Section~\ref{sec:laplacegauss}}

\subsection{Proof of Corollary~\ref{coro:laplace}}
\label{sec-prooflap}
\begin{proof}
  We apply Proposition~\ref{prop:basiccompo} (resp. Proposition~\ref{prop:basiccompoetaW}) and Theorem~\ref{thm:tsystem}, noting that the conclusion of Proposition~\ref{prop:basiccompo} still holds if we only assume that the product of the two determinants
  $\detAV(\s_1,\ldots,\s_{2\M})$ and $\detBV(t,\s_1,\ldots,\s_\M)$ is (strictly) positive.
   
  The condition $\int_{\domobs} (1+|\s|)^{4\M}e^{-2c\s}\d\Ps(s)<+\infty$ ensures that~\eqref{eq:derivint} holds for $0\leq k\leq 2\M$.
  As noted above, the exponential kernel is \textit{extended totally positive}, if $0<\pos_1<\ldots <\pos_\M$ (resp. $0<\pos_0$), and $0<\s_1< \ldots < \s_{2\M}$, we get 
\begin{align*}
  0&<  (-1)^{{\M(2\M-1)}}  \begin{mydet}
   \psi(\pos_1,\s_1) & \ldots & \partial_1\psi(\pos_\M,\s_{1})\\ 
   \vdots & & \vdots\\
  \psi(\pos_1,\s_{2\M}) & \ldots & \partial_1\psi(\pos_\M,\s_{2\M}) 
  \end{mydet}\\
   &= (-1)^{{\M(2\M-1)}}\detAV(\s_1,\ldots,\s_{2\M}),\\
  \mbox{(resp.)}\quad  0&< (-1)^{{\M(2\M-1)}}  \begin{mydet}
      \psi(\pos_0,\s_1) & \ldots & (\partial_1)^{2\M-1}\psi(\pos_0,\s_{1})\\ 
   \vdots & & \vdots\\
      \psi(\pos_0,\s_{2\M}) & \ldots & (\partial_1)^{2\M-1}\psi(\pos_0,\s_{2\M}) 
    \end{mydet}\\
    &= (-1)^{{\M(2\M-1)}}\detAW(\s_1,\ldots,\s_{2\M}).
\end{align*}
As for the other determinant, introducing $\s_0=0$, we obtain
\begin{align*}
  0&<\frac{(-1)^{{\M(2\M+1)}}}{\prod_{i=1}^{\M}(t-\pos_i)^2} \begin{mydet}
    \psi(t,\s_0) & \psi(t,\s_1) & \ldots & \psi(t,\s_{2\M})\\ 
    \psi(\pos_1,\s_0) & \partial\psi(\pos_1,\s_1) & \ldots & \psi(\pos_1,\s_{2\M})\\ 
    \partial_1\psi(\pos_1,\s_0) & \partial_1\partial\psi(\pos_1,\s_1) & \ldots & \partial_1\psi(\pos_1,\s_{2\M})\\ 
    \vdots & \vdots & & \vdots\\
  \psi(\pos_\M,\s_0) & \partial\psi(\pos_{\M},\s_1) & \ldots & \psi(\pos_{\M},\s_{2\M})\\ 
  \partial_1\psi(\pos_{\M},\s_0) & \partial_1\partial\psi(\pos_{\M},\s_1) & \ldots & \partial_1\psi(\pos_{\M},\s_{2\M})
  \end{mydet}\\
  &=(-1)^{{\M(2\M+1)}}\detBV(t,\s_1,\ldots,\s_\M)
\end{align*}
The same holds for $t=\pos_i$, with the usual extension.
\end{proof}

\subsection{Proof of Proposition~\ref{prop:fullgaussian}}
\label{sec:prooffullgauss}
\begin{proof}[Proof of Proposition~\ref{prop:fullgaussian}]
  The correlation function is given by $\Co(\pos,\pos')=e^{\frac{1}{2}(\pos-\pos')^2}\Co(0,0)$, and $\partial_2\Co(\pos,\pos')=(\pos'-\pos)e^{\frac{1}{2}(\pos-\pos')^2}\Co(0,0)$.
  
  Setting $K(z,z')=e^{\frac{1}{2}(z-z')^2}$ we see that, up to a positive factor, $\det(\Gamma_{\bpos}^*\Gamma_{\bpos})$ is exactly the determinant $\det(K(z_i,z'_j)$ for $z_{2i-1}=z_{2i}=\pos_i$ (resp.~$z'_{2i-1}=z'_{2i}=\pos_i$) with the replacement introduced described in \eqref{eq:replaceET} in $z$ and $z'$. Hence, by the \textit{extended total positivity} of the Gaussian kernel (see~\cite{karlin1968total}), $\det(\Gamma_{\bpos}^*\Gamma_{\bpos})>0$ and $\Gamma_{\bpos}$ has full rank.

  Now, dropping the constant factors $\Co(0,0)$, let us write 
     \begin{align*}
       \begin{matrix}
         u_0(t):=1, &  u_1(t):=e^{-(t-\pos_1)^2/2}, & \cdots & u_{2\M-1}(t):=e^{-(t-\pos_\M)^2/2},\\
                    & u_2(t):=(t-\pos_1)e^{-(t-\pos_1)^2/2}, &  \cdots & u_{2\M}(t):=(t-\pos_{\M})e^{-(t-\pos_\M)^2/2}.
       \end{matrix}
   \end{align*}
  By the Leibniz formula, we note that any function of the form
\begin{align*}
\t\mapsto \sum_{i=0}^{2\M} b_i u_i(t)
\end{align*}
  has the same roots (with the same multiplicites) as the function 
  \begin{align}\label{eq:gausslap}
    \t\mapsto \sum_{i=1}^{\M} \left[(b_{2i-1}+b_{2i}(\t-\pos_i))e^{\frac{1}{2}\pos_i^2}\right]e^{\pos_it} +b_0 e^{\frac{1}{2}t^2},
  \end{align}
If $b_0=0$, it is known that~\eqref{eq:gausslap} has at most $2\M-1$ zeros (as a consequence of the extended total positivity of the Gaussian, see also~\cite[Ex. V.75]{polya_problems_1972}). If $b_0\neq 0$, it follows from Lemma~\ref{lem:polygauss} that~\eqref{eq:gausslap} has at most $2\M$ zeros. As a result, the function
\begin{align}\label{eq:gausratio}
\detV:  t\longmapsto \frac{2}{\prod_{i=1}^\M (t-\pos_i)^2} \begin{mydet}
        u_0(t) & u_1(t) & \cdots &u_{2\M}(t)\\
        u_0(\pos_1) & u_1(\pos_1) & \cdots &u_{2\M}(\pos_1)\\
        u_0'(\pos_1) & u_1'(\pos_1) & \cdots &u_{2\M}'(\pos_1)\\
        \vdots & \vdots & & \vdots\\
        u_0(\pos_{\M}) & u_1(\pos_{\M}) & \cdots &u_{2\M}(\pos_{\M})\\
        u_0'(\pos_{\M}) & u_1'(\pos_{\M}) & \cdots &u_{2\M}'(\pos_{\M})
    \end{mydet},
  \end{align}
  which is continuous on $\dompos$, does not vanish. To evaluate its sign, we let $t\to+\infty$, noting that $u_i(t)\to 0$ for $1\leq i\leq 2\M$, and we expand along the first row
\begin{align*}
  \detV(t)&= \frac{2}{\prod_{i=1}^\M (t-\pos_i)^2}\left(
\begin{mydet}
        u_1(\pos_1) & \cdots &u_{2\M}(\pos_1)\\
         u_1'(\pos_1) & \cdots &u_{2\M}'(\pos_1)\\
         \vdots & & \vdots\\
         u_1(\pos_{\M}) & \cdots &u_{2\M}(\pos_{\M})\\
         u_1'(\pos_{\M}) & \cdots &u_{2\M}'(\pos_{\M})
\end{mydet} + o(1)
\right)\\
&>0 \mbox{ for $t$ large enough,}
\end{align*}
 by the extended total positivity of the Gaussian.
 As a result, $\detV$ is strictly positive on $\RR$, and the (NDSC) holds.

  The second conclusion is obtained by applying similar arguments to $\det(\Fdn^*\Fdn)$ and to the functions 
\begin{align}
  \begin{matrix}
    u_0(t):=1, &  u_1(t):=e^{-(t-\pos_0)^2/2},& u_2(t)=\tH_{2}(t)e^{-(t-\pos_0)^2/2}, & \cdots & \tH_{2\M-1}(t)e^{-(t-\pos_0)^2/2},
       \end{matrix}
\end{align}
where $\tH_{k}$ is defined by $\tH_k(u)=e^{u^2}\frac{\d^k}{\d u^k}(e^{-u^2})$, that is (up to a $(-1)^k$ factor) the $k$-th Hermite polynomial (note that each $\tH_k$ has degree $k$). We skip the detail for brevity. 
\end{proof}

\subsection{Proof of Proposition~\ref{prop:samplebesque}}
\label{sec:proofsamplebesque}
\begin{proof}
  The autocorrelation functions corresponding to the Lebesgue measure and $\Pns$ are defined by
  \begin{align*}
    \Co(\pos,\pos')&=\int_{\RR}  e^{-(\pos-\s)^2}e^{-(\pos'-\s)^2}\d \s, \\
    \Con(\pos,\pos')&=\int_{\RR}  e^{-(\pos-\s)^2}e^{-(\pos'-\s)^2}\d\Pns(\s),\\
    \intertext{so that, denoting  by $\tH_k$  the $k$-th Hermite polynomial as above,}
    (\partial_1)^{k}(\partial_2)^{\ell}\Co(\pos,\pos')&=\int_{\RR} \tH_k(\pos-\s)\tH_\ell(\pos'-\s) e^{-(\pos-\s)^2}e^{-(\pos'-\s)^2}\d\s,\\
    (\partial_1)^{k}(\partial_2)^{\ell}\Con(\pos,\pos')&=\int_{\RR} \tH_k(\pos-\s)\tH_\ell(\pos'-\s) e^{-(\pos-\s)^2}e^{-(\pos'-\s)^2}\d\Pns(\s).
  \end{align*}
  Denoting by $\Gamma_\bpos$ (resp. $\Gamma_\bpos^n$) the linear operators defined by~\eqref{eq:defgamma}, we consider the vanishing derivatives precertificates, 
  \begin{align}
    \etaVn =\sum_{i=1}^{\M}\alpha_i^n\left(\Con(\cdot,\pos_i)+\beta_i^n\partial_2\Con(\cdot,\pos_i)\right),\qquad &
    \etaV =\sum_{i=1}^{\M}\alpha_i\left(\Co(\cdot,\pos_i)+\beta_i\partial_2\Co(\cdot,\pos_i)\right)\\
    \mbox{where}
    \begin{pmatrix}
      \alpha^n\\\beta^n 
    \end{pmatrix}=({\Gamma_\bpos^n}^*\Gamma_\bpos^n)^{-1}\begin{pmatrix}
      1\\0\\\vdots\\ 1\\0
    \end{pmatrix},\qquad&\begin{pmatrix}
      \alpha\\\beta 
    \end{pmatrix}=({\Gamma_\bpos}^*\Gamma_\bpos)^{-1}\begin{pmatrix}
      1\\0\\\vdots\\ 1\\0
    \end{pmatrix}
  \end{align}
  By Proposition~\ref{prop:fullgaussian} above, $\etaV(t)<1$ for $t\in \dompos\setminus\{\pos_i\}_{i=1}^\M$, $\etaV(\pos_i)=1$, and $\etaV''(\pos_i)<0$ for all $1\leq i\leq \M$. We prove that the same holds for $\etaVn$ by uniform convergence of the derivatives.

We note that as $n\to+\infty$
\begin{align*}
  \norm{({\Gamma_{\bpos}^n}^*\Gamma_{\bpos}^n)^{-1}-(\Gamma_{\bpos}^*\Gamma_{\bpos})^{-1}} \leq \frac{\norm{\Id-(\Gamma_{\bpos}^*\Gamma_{\bpos})^{-1}({\Gamma_{\bpos}^n}^*\Gamma_{\bpos}^n)}}{1-\norm{\Id-(\Gamma_{\bpos}^*\Gamma_{\bpos})^{-1}({\Gamma_{\bpos}^n}^*\Gamma_{\bpos}^n)}}\longrightarrow 0,
\end{align*}
therefore $\alpha^n\to \alpha$ and $\beta^n\to \beta$.
Moreover, \eqref{eq:cvgauss} implies that for $k\in \{0,1,2\}$ and $\ell\in \{0,1\}$,
\begin{align*}
\sup_{x\in\dompos}\left|(\partial_1)^k (\partial_2)^\ell\Con(\pos,\pos_i)-(\partial_1)^k (\partial_2)^\ell\Co(\pos,\pos_i)\right|\to 0,
\end{align*}
which implies that $\sup_{x\in\dompos}|\etaVn{}^{(k)}(x)-\etaV^{(k)}(x)|\to 0$. Observing that  that there is some $r>0$, some $\epsilon>0$, such that 
\begin{align*}
  \etaV&\leq 1-\varepsilon\quad \mbox{in $\dompos\setminus \bigcup_{i=1}^\M(\pos_i-r,\pos_i+r)$,}\\
  \etaV{}''&<0\quad \mbox{in $\dompos\cap \left(\bigcup_{i=1}^\M(x_i-r,x_i+r)\right)$,}\\
  \etaV&<1 \quad \mbox{in $\dompos\cap \left(\bigcup_{i=1}^\M(x_i-r,x_i+r)\setminus\{\pos_i\}\right)$.}\\
\end{align*}
we conclude by standard arguments of uniform convergence that the same holds for $\etaVn$ for $n$ large enough. 
This ensures the Non-Degenerate Source Condition holds for $m_0$ with $\Pns$.

As for the similar statement concerning the $2\M-1$ Non-Degeneracy of $\pos_0$, we skip the proof as it is the same pattern, involving $\etaW$ and the matrices $\Fdn$ instead.
\end{proof}

\subsection{Proof of Proposition~\ref{prop:confined}}
\label{sec:proofconfined}

We begin with a continuous extension lemma for rescaled determinants.
\begin{lemma}\label{sec:supervdm}
  Let $\{u_1,\ldots u_n\}\subset \Cder{n-1}(\RR)$.

Then the function 
\begin{align}
  U: (\s_1,\ldots \s_n)\longmapsto \frac{1}{\prod_{1\leq i<j\leq n}(\s_j-\s_i)}
  \begin{mydet}
    u_1(\s_1) & \cdots & u_n(\s_1)\\
    \vdots & & \vdots\\
    u_1(\s_n) & \cdots & u_n(\s_n)
  \end{mydet}
\end{align}
has a continuous extension to $\RR^n$. Moreover, there exists some constant $\gamma>0$ such that 
\begin{align}
\forall t\in \RR,\quad  U(s,\ldots,s)=\gamma  \begin{mydet}
      u_1(s) & \cdots & u_n(s)\\
      u_1'(s) & \cdots & u_n'(s)\\
    \vdots & & \vdots\\
  u_1^{(n-1)}(s) & \cdots & u_n^{(n-1)}(s)
\end{mydet}.
\end{align}
\end{lemma}
\begin{proof}
  We prove the continuity by operations on rows in the determinant. As we operate on rows only, we only represent the $j$-th column.
Applying a Taylor expansion and substracting the first row to its successors, then the second one to its successors,
  \begin{align*}
    \begin{mydet}
      u_j(\s_1)\\
      \vdots\\
      u_j(\s_n)
    \end{mydet}_{1\leq j\leq n}&=\left(\prod_{i>2}(\s_i-\s_1)\right)\int_0^1\begin{mydet}
      u_j(\s_1)\\
      u_j'(\s_1(1-\theta_1)+\s_2\theta_1)\\
      \vdots\\
      u_j'(\s_1(1-\theta_1)+\s_n\theta_1)
    \end{mydet}_{1\leq j\leq n}\d \theta_1\\
    &= \left(\prod_{i>1}(\s_i-\s_1)\right)\left(\prod_{i>2}(\s_i-\s_2)\right)\\
    &\quad \times \int_{[0,1]^2}\theta_1^{n-2}\begin{mydet}
      u_j(\s_1)\\
      u_j'(\s_1(1-\theta_1)+\s_2\theta_1)\\
      u_j''(\s_1(1-\theta_1)+\s_2\theta_1(1-\theta_2)+\s_3\theta_1\theta_2)\\
      \vdots\\
      u_j''(\s_1(1-\theta_1)+\s_2\theta_1(1-\theta_2)+\s_n\theta_1\theta_2)\\
    \end{mydet}_{1\leq j\leq n}\d \theta_1\d \theta_2.
  \end{align*}
  Iterating this procedure, we see that we may write
\begin{align}
     \begin{mydet}
      u_j(\s_1)\\
      \vdots\\
      u_j(\s_n)
    \end{mydet}_{1\leq j\leq n}&= \left(\prod_{k>l}(\s_k-\s_l)\right)\int_{[0,1]^{n-1}}\theta_1^{\alpha_1}\ldots \theta_{n-2}^{\alpha_{n-2}}\\
    & \times 
    \begin{mydet}
      u_j(\s_1)\\
      u_j'(\s_1(1-\theta_1)+\s_2\theta_1))\\
      u_j''(g_{3}(\s_1,\s_2, \s_3,\theta_1,\theta_2))\\
      \vdots\\
      u_j^{(n-1)}(g_{n-1}(\s_1,\ldots, \s_n,\theta_1,\ldots,\theta_{n-1}))\\
    \end{mydet}_{1\leq j\leq n}\-\-\d \theta_1\ldots \d \theta_{n-1}
\end{align}
for some exponents $\alpha_1,\ldots,\alpha_{n-2}$ in $\NN$, and with \begin{align*}
  g_{k}(\s_1,\ldots, \s_k,\theta_1,\ldots,\theta_{k-1})&=\s_1(1-\theta_1) + \s_2\theta_1(1-\theta_2) +\cdots \\
  &\quad +\s_{k-1} \theta_1\theta_2\ldots \theta_{k-2}(1-\theta_{k-1})+ \s_k \theta_1\theta_2\ldots \theta_{k-2}\theta_{k-1}
\end{align*}
which describes some convex combination of $\s_1,\ldots, \s_k$, for $1\leq k\leq n-1$. 

This yields the claimed continuity and value at $(s,\ldots,s)$.
\end{proof}

Now, we are in position to prove Proposition~\ref{prop:confined}.

\begin{proof}[Proof of Proposition~\ref{prop:confined}]
  First, we note that $\detAV(\s_1,\ldots,\s_{2\M})>0$ (resp.~$\detAW(\s_1,\ldots,\s_{2\M})>0$) always holds by the extended total positivity of the Gaussian kernel.
  Next, we show that $\detBV(t,\s_1,\ldots,\s_{2\M})>0$  holds.

Now, we introduce a renormalized version of $\detBV$, 
\begin{align}
  F(t,\s_1,\ldots,\s_{2\M}) &= \frac{\abs{t}^{2\M}}{\prod(t-\pos_i)^2}\frac{1}{\prod_{i<j}(\s_j-\s_i)} \begin{mydet}
  1 & \psi(t,\s_1) & \cdots & \psi(t,\s_{2\M})\\
  1 & \psi(\pos_1,\s_1) & \cdots & \psi(\pos_1,\s_{2\M})\\
  0 & \partial_1 \psi(\pos_1,\s_1) & \cdots & \psi(\pos_1,\s_{2\M})\\
  \vdots & & \vdots \\
  1 & \psi(\pos_{\M},\s_1) & \cdots & \psi(\pos_{\M},\s_{2\M})\\
  0 & \partial_1 \psi(\pos_{\M},\s_1) & \cdots & \psi(\pos_{\M},\s_{2\M})\\
  \end{mydet}
\end{align}
  
Let $r'>0$. By Lemma~\ref{sec:supervdm}, which is a refined version of Lemma~\ref{lem:vandermonde}, we note that $H$ is continuous on $\RR\times [\s_*-r',s_*+r']$. Now, we prove that it can be extended by continuity to the compact set $\overline{\RR}\times [\s_*-r',\s_*+r']$, where $\overline{\RR}=[-\infty,+\infty]$, equipped with a suitable metric\footnote{For instance with $d(t,t')=|\arctan(t)-\arctan(t')|$.}.  

From the proof of  Lemma~\ref{sec:supervdm}, we know that $H$ may be written as 
\begin{align}
  F(t,\s_1,\ldots,\s_{2\M}) &= \frac{\abs{t}^{2\M}}{\prod(t-\pos_i)^2} \int_{[0,1]^{2\M-1}}\theta_1^{\alpha_1}\ldots \theta_{2\M-2}^{\alpha_{2\M-2}}\-\-H(t,(\s_j)_{1\leq j\leq 2\M},(\theta_j)_{1\leq j\leq 2\M-1})\d\theta_1\ldots\d\theta_{2\M-1}. 
\end{align}
where $H$ is the determinant of the matrix whose first row is:
\begin{align}\label{eq:firstrow}
\begin{pmatrix}
  1 & \psi(t,\s_1) & \partial_2\psi(t,g_1((\s_j)_{1\leq j\leq 2},\theta_1)) & \cdots  & (\partial_2)^{2\M-1}\psi(t,g_{2\M-1}((\s_j)_{1\leq j\leq 2\M-1},(\theta_j)_{1\leq j\leq 2\M-2})) 
\end{pmatrix}\\
\intertext{and the rows $1+2i$ and $2+2i$ are}
\begin{pmatrix}
  1 & \psi(\pos_i,\s_1) & \partial_2\psi(\pos_i,g_1((\s_j)_{1\leq j\leq 2},\theta_1)) & \cdots  & (\partial_2)^{2\M-1}\psi(\pos_i,g_{2\M-1}((\s_j)_{1\leq j\leq 2\M-1},(\theta_j)_{1\leq j\leq 2\M-2}))\\ 
  0 & \partial_1\psi(\pos_i,\s_1) & \partial_1\partial_2\psi(\pos_i,g_1((\s_j)_{1\leq j\leq 2},\theta_1)) & \cdots  & \partial_1(\partial_2)^{2\M-1}\psi(\pos_i,g_{2\M-1}((\s_j)_{1\leq j\leq 2\M-1},(\theta_j)_{1\leq j\leq 2\M-2}))
\end{pmatrix}\label{eq:otherrow}
\end{align}
Here $g_{k}(\s_1,\ldots,\s_{k},\theta_1,\ldots, \theta_{k-1}))$ is a convex combination of $\s_1,\ldots,\s_k$.

Since $\{\s_j\}_{j=1}^{2\M}\subset [s_*-r',s_*+r']$, we see that, as $\abs{t}\to+\infty$, the first row~\eqref{eq:firstrow} converges uniformly to the row $\begin{pmatrix}
  1 & 0 & 0 & \cdots  & 0
\end{pmatrix}$. 
As a result, since $\frac{\abs{t}^{2\M}}{\prod(t-\pos_i)^2}\sim 1$ and the other rows~\eqref{eq:otherrow} are bounded, we obtain that $F(t,\s_1,\ldots,\s_{2\M})$ converges uniformly in $(\s_j)_{j=1}^{2\M}$ towards 
\begin{align}
  F(\pm\infty,\s_1,\ldots,\s_{2\M}) &\eqdef  \int_{[0,1]^{2\M-1}}\theta_1^{\alpha_1}\ldots \theta_{2\M-2}^{\alpha_{2\M-2}}\-\-H(\pm\infty,(\s_j)_{1\leq j\leq 2\M},(\theta_j)_{1\leq j\leq 2\M-1})\d\theta_1\ldots\d\theta_{2\M-1}. 
\end{align}
where $H(\pm\infty,(\s_j)_{1\leq j\leq 2\M},(\theta_j)_{1\leq j\leq 2\M-1})$ is the determinant of the matrix whose first row is $\begin{pmatrix}
  1& 0& \cdots & 0
\end{pmatrix}$ and the other rows are given by~\eqref{eq:otherrow}.

As a result $F$ is continuous on the compact set $\overline{\RR}\times [\s_*-r',\s_*+r']$, hence uniformly continuous. For all $\varepsilon>0$, there exists $r>0$ such that $\max_{i}|\s_i-\s_*|\leq r$ implies 
\begin{align}\label{eq:epsilon}
 | F(t,\s_1,\ldots,\s_{2\M}) - F(t,\s_*,\ldots,\s_*) |\leq \varepsilon.
\end{align}

Now, there is some constant $\gamma>0$ such that 
\begin{align*}
  F(t,\s_*,\ldots,\s_*)&=  \frac{\abs{t}^{2\M}}{\prod(t-\pos_i)^2} \gamma 
  \begin{mydet}
    1 & \psi(t,\s_*) & \partial_2\psi(t,\s_*) & \cdots  & (\partial_2)^{2\M-1}\psi(t,\s_*) \\
    1 & \psi(\pos_1,\s_*) & \partial_2\psi(\pos_1,\s_*) & \cdots  & (\partial_2)^{2\M-1}\psi(\pos_1,\s_*)\\ 
  0 & \partial_1\psi(\pos_1,\s_*) & \partial_1\partial_2\psi(\pos_1,\s_*) & \cdots  & \partial_1(\partial_2)^{2\M-1}\psi(\pos_1,\s_*)\\
    \vdots & \vdots & \vdots  & & \vdots\\
    1 & \psi(\pos_\M,\s_*) & \partial_2\psi(\pos_\M,\s_*) & \cdots  & (\partial_2)^{2\M-1}\psi(\pos_\M,\s_*)\\ 
  0 & \partial_1\psi(\pos_\M,\s_*) & \partial_1\partial_2\psi(\pos_\M,\s_*) & \cdots  & \partial_1(\partial_2)^{2\M-1}\psi(\pos_M,\s_*)
  \end{mydet}
\end{align*}
We have already met similar determinants in the proof of Proposition~\ref{prop:fullgaussian}. They characterize the number of roots (including multiplicity) of a function of the form
\begin{align}
  t\longmapsto  b_0 + (b_{1,0} +b_{1,1}t + \ldots +b_{1,2\M-1}t^{2\M-1})e^{-(t-\s_*)^2/2},
\end{align}
which is at most $2\M$ by~\cite[Lemma 2.7]{schiebinger2015superresolutionjournal}. Therefore, arguing as in Proposition~\ref{prop:fullgaussian}, it is possible to show that this determinant is positive, including for $t=\pos_i$. Moreover $F(\pm\infty,\s_*,\ldots,\s_*)>0$ by extended total positivity of the Gaussian. As a result, the function  $t\mapsto F(t,\s_*,\ldots,\s_*)$ has a lower bound $\ell>0$ on $[-\infty,+\infty]$.

In~\eqref{eq:epsilon} it is sufficient to choose $\varepsilon<\ell/2$, and for the corresponding $r>0$, $\max_i |\s_i-\s_*|\leq r$ implies $F(t,\s_1,\ldots, \s_{2\M})>0$, hence $\detBV(t,\s_1,\ldots,\s_{2\M})>0$ for $\s_*-r\leq \s_1<\ldots <\s_{2\M}\leq \s_*+r$.
 We conclude by applying Proposition~\ref{prop:basiccompo}. The case of $\pos_0$ being $2\M-1$ non-degenerate is similar.
\end{proof}

\bibliographystyle{plain}
\bibliography{refs}
\end{document}